\documentclass[11pt]{article}

\RequirePackage[l2tabu, orthodox]{nag}
\usepackage{color} 
\usepackage[usenames,dvipsnames]{xcolor}
\usepackage{amsrefs,amsthm,amssymb,amsmath}
\usepackage{graphicx}
\usepackage{enumerate}

\newcommand{\R}{\mathbf{R}}
\newcommand{\N}{\mathbf{N}}

\newcommand{\bP}{\mathbf{P}}
\newcommand{\mN}{\mathcal{N}}

\newcommand{\bA}{\mathbf{A}}
\newcommand{\bB}{\mathbf{B}}
\newcommand{\bb}{\mathbf{b}}
\newcommand{\rt}{\bullet}
\newcommand{\bm}{\bar{\mu}}
\newcommand{\tmu}{\widetilde{\mu}}
\newcommand{\T}{\mathcal{T}}
\newcommand{\mF}{\mathcal{F}}
\newcommand{\mG}{\mathcal{G}}
\newcommand{\mC}{\mathcal{C}}
\newcommand{\mK}{\mathcal{K}}
\newcommand{\tX}{\widetilde{X}}
\newcommand{\tc}{\widetilde{c}}
\newcommand{\cP}{\mathcal{P}}
\newcommand{\cQ}{\mathcal{Q}}
\newcommand{\mP}{P}
\newcommand{\eps}{\varepsilon}

\newcommand{\const}{\mathop{\mathrm{const}}} 
\newcommand{\ddelta}{\mathrm{Dirac}}

\newcommand{\Fbm}{F_{\bm}}
\newcommand{\kal}{\kappa_{\alpha}}
\newcommand{\kone}{\kappa_{(1-\alpha)}}

\newtheorem{thm}{Theorem}
\newtheorem{lem}{Lemma}
\newtheorem{prop}{Proposition}

\theoremstyle{definition}
\newtheorem{dfn}{Definition}

\newtheorem{assumption}{}

\theoremstyle{remark}
\newtheorem{rem}{Remark}

\usepackage[margin=2.5cm]{geometry}

\usepackage{indentfirst}
\usepackage{microtype}

\usepackage[colorlinks=true,linkcolor=blue,citecolor=magenta]{hyperref}

\usepackage{chngcntr}
\counterwithin{paragraph}{section}

\title{Cut-off method for endogeny of recursive tree processes}
\date{\today}
\author{Victor Kleptsyn\footnote{CNRS, Institut de Recherches Math\'ematiques de Rennes, Campus Beaulieu, 35042 Rennes, France.
\newline E-mail: victor.kleptsyn@univ-rennes1.fr}, Michele Triestino\footnote{
Institut de Math\'ematiques de Bourgogne,
9 av.~Alain Savary,
21000 Dijon, France. \newline E-mail: michele.triestino@u-bourgogne.fr}}

\begin{document}

\maketitle

\begin{abstract}
Given a solution to a recursive distributional equation, a natural (and non-trivial) question is whether the corresponding recursive tree process is \emph{endogenous}. 
That is, whether the random environment almost surely defines the tree process.
We propose a new method of proving endogeny, which applies to various processes. As explicit examples, we establish endogeny of the random metrics on non-pivotal hierarchical graphs defined by multiplicative cascades and of mean-field optimization problems as the mean-field matching and travelling salesman problems in pseudo-dimension $q>1$.
\footnote{\textbf{MSC\textup{2010}:} Primary 60E05, 60B10, 60E15. Secondary 81T20, 82B44, 90C27.
}
\end{abstract}

\section{Introduction}
\emph{Recursive distributional equations} (RDEs) and the associated \emph{recursive tree processes} (RTPs) occur in various problems, mainly related to branching processes, random geometry, statistical physics and combinatorial optimization, as also probabilistic analysis of algorithms.
We briefly recall here some generalities, referring the reader to the excellent survey~\cite{RDE} for more details.

Informally speaking, recursive tree processes are analogues of Markov processes, but with each value ``taking as input'' several ``preceding'' values, forming thus a tree of dependences (whence the name). More precisely, a RTP is a rooted tree $(\T,\rt)$ of random variables $\{(X_t,\xi_t)\}_{t\in\T}$, where at each vertex $t\in \T$ the value $X_t$ is related to the values $X_{t_i}$'s at its children and to the random variable~$\xi_t$ via a given function $R$:
\[
X_t=R(X_{t_1},\dots,X_{t_d};\xi_t);
\]
in addition, natural independence conditions on $X_t$'s and $\xi_t$'s are required. We give the precise definition below, see \S\,\ref{par:IRTP}.

Now, assume that the laws of all the $\xi_t$'s are given according to the same probability measure~$m$ (this assumption will stand throughout the whole paper). Then it is natural to ask whether the laws of all the~$X_t$'s coincide, too. If this is the case, the RTP is called \emph{invariant}, and the law of~$X_t$ is a solution to the RDE defined by the function $R$:
\[
X\stackrel{d}{=}R(X_1,\ldots,X_d;\xi).
\]

Given an invariant RTP, it is highly interesting to find out whether the values of the $X_t$'s are (almost surely) defined by the random environment $\{\xi_t\}_{t\in \T}$, that is, whether the conditional distribution of the $\{X_t\}$'s given the $\{\xi_t\}$'s is almost surely a Dirac measure. If this is the case, according to Aldous and Bandyopadhyay \cite{RDE}, the RTP is called~\emph{endogenous}.

Even existence of invariant RTPs (or, what is the same, existence of a solution to a RDE) is a difficult question in general. We again refer the reader to~\cite[\S~2.2]{RDE} for an overview of known methods (among which, the most widely used are monotonicity for the compact case and contraction in more general settings). 

\smallskip

In~\cite{KKT}, in collaboration with M.~Khristoforov, we have proposed a \emph{cut-off method} to construct \emph{solutions} of monotone RDEs in a non-compact setting. We have applied this method in the particular case of random metrics associated with multiplicative cascades on hierarchical graphs (we recall the details below, in~\S\,\ref{s:background}). 
Moreover, the method is sufficiently flexible to be applied to other situations (for instance, $(\min,+)$-type RDEs with some additional conditions imposed). 

The aim of this paper is to show that the cut-off method can be modified also to prove \emph{endogeny} of certain RTPs. As a first explicit example, we establish 
endogeny of the tree processes associated with random metrics defined via multiplicative cascades on non-pivotal hierarchical graphs. This has as rather reassuring consequence that the random metric spaces that were defined in~\cite{KKT} are measurable functions of the cascade (cf.~\cite[\S\,3.4]{KKT}).

For the moment, we state our main theorem with a ``black box'' of assumptions, postponing the listing  till \S\,\ref{s:assumptions}:
\begin{thm}[Endogeny via cut-off]\label{t:main}
Suppose that the law $m$ and the relation function $R$ satisfy assumptions~\ref{a:incr}--\ref{a:superexp}, listed in \S\,\ref{s:assumptions}, and let~$\bm$ be a solution to the corresponding RDE. Then the invariant RTP associated with $\bm$ is endogenous.
\end{thm}

Its applicability for the multiplicative cascades on non-pivotal hierarchical graphs is guaranteed by the second theorem:
\begin{thm}\label{t:applying}
Assumptions~\ref{a:incr}--\ref{a:superexp} are satisfied for the RDEs \eqref{eq:Gamma-log} corresponding to random metrics associated with multiplicative cascades on non-pivotal hierarchical graphs with a Gaussian measure~$m$ (that is, with the log-normal random rescaling law).
\end{thm}
\begin{rem}
As in~\cite[\S\,1.5, Rem.~2]{KKT}, the statement of Theorem~\ref{t:applying} holds also provided that the measure~$m$ is absolutely continuous with everywhere positive continuous density and has normal tails at $\pm \infty$.
\end{rem}

Interestingly, different approaches by cut-off have been used to find solutions to the RDEs that are associated with problems of mean-field combinatorial optimization (see Section \ref{s:general}, as well as \cite{zeta2}, \cite[\S\,7]{RDE}): the mean-field approximation for minimal matching in pseudo-dimension $0<q<\infty$ \cite{salez,W-match,larsson},
and for the minimum weight $k$-factor problem, which corresponds to the mean-field travelling salesman problem when $k=2$ \cite{W-TSP,PW}. For the latter, we point out that \cite{K} makes no use of cut-off to show existence of the solution, although many ideas are in the spirit of \cite{KKT}, when proving that the solution is a global attractor. Relying on these works, we apply our main result to show the endogeny. For the minimal matching, the case $q=1$ was already solved in \cite{B}; here we use the results in \cite{salez,W-match} to treat the problem of pseudo-dimension $q>1$.
This completes Aldous' program for the rigorous analysis of physical predictions for these models (see \cite[\S~7.5]{RDE} and \cite{BP,KS}).

\begin{thm}\label{t:TSP}
Assumptions~\ref{a:incr}--\ref{a:superexp} are satisfied for the RDE \eqref{eq:TSPk} corresponding to the minimum weight $k$-factor problem ($k\ge 1$) in pseudo-dimension $q>1$.

This includes the mean-field minimal matching problem ($k=1$) and the mean-field travelling salesman problem ($k=2$).
\end{thm}

The structure of the paper is as follows: in Section~\ref{s:defs} we introduce RDEs and list the assumptions for Theorem~\ref{t:main}; in Section~\ref{s:cut-off} we review the cut-off method, explaining how to adapt it to prove endogeny. The key construction is carried out in Section~\ref{s:construction}. In Section~\ref{s:metrics} we show endogeny for the random metric problem (Theorem~\ref{t:applying}), and in Section~\ref{s:general} for the mean-field optimization problems (Theorem~\ref{t:TSP}).  

\section{Definitions and assumptions}\label{s:defs}

\paragraph{Recursive distributional equations --} Denote by $\cP$ the space of Radon probability measures on the compactified real line $[-\infty,\infty]$, equipped with the topology of weak convergence of measures. For the sequel, we choose and fix a metric on~$\cP$ that defines this topology (it exists, due to the compactness of $[-\infty,+\infty]$). 

Let a probability measure $m\in \cP$ be given, as well as a function $R:[-\infty,\infty]^d\times [-\infty,\infty] \to [-\infty,\infty]$ of $d+1$ variables; the number of variables $d$ can be either finite or infinite. We put them in correspondence to a transformation $\Phi$ of $\cP$, defined by
\begin{equation}\label{eq:RDE}
\Phi[\mu]=R_*(\mu^{\otimes d}\otimes m).
\end{equation}
In other words, we define $\Phi[\mu]$ to be the law of the random variable
\[
R(X_1,\ldots X_d; \xi),
\]
where the $X_i$'s have law $\mu$, $\xi$ has law $m$ and they are all independent.

\begin{dfn} 
A law $\mu$ satisfies the \emph{recursive distributional equation} associated with the relation function~$R$ and the law~$m$ if it is a fixed point for the operator $\Phi$. Alternatively, we say that the measure $\mu$ is \emph{stationary}.
\end{dfn}

Most of the time we are interested in the case where the function $R$ takes finite values as soon as all its variables are finite, and the measure $m$ belongs to the subset
$\cP_0$ of measures supported on $\R$ (that is, not charging $\pm \infty$), and after all, we want to consider 
stationary measures from~$\cP_0$, too. However, it turns out to be quite useful to allow random variables to 
take infinite values during the intermediate steps of the construction (see Section~\ref{s:cut-off} below).

\paragraph{Invariant recursive tree processes --}\label{par:IRTP} Given a function $R$ and a measure $m$ as above, it is convenient to uncover the tree structure behind the RDE. 
We follow the presentation given in~\cite[\S\,4.1]{KKT}; the reader may also consult \cite[\S\,2.3]{RDE} as a more classical reference.

Consider the rooted $d$-ary tree $(\T,\rt)$; we denote by $\|\,\cdot\,\|:\T\to\N$ the distance of a vertex to the root. Given a solution $\mu$ to the RDE associated with $R$ and $m$, we define the \emph{invariant recursive tree process} as a tree of random variables $\{(X_t,\xi_{t})\}_{t\in\T}$ such that
\begin{itemize}
\item for every $t\in\T$ the law of $\xi_t$ is $m$,
\item for every $t\in \T$ the law of $X_t$ is $\mu$,
\item for every $t\in \T$ the equality
\begin{equation}\label{vertex}
X_t=R(X_{t_1},\ldots,X_{t_d};\xi_t)
\end{equation}
holds, where the $t_i$'s are the $d$ children of $t$ in $\T$,
\item for every $n\in\N$ the random variables $\{\xi_t\}_{\|t\|< n}$ and $\{X_t\}_{\|t\|=n}$ are independent altogether (in particular the random variables $\{\xi_t\}_{t\in\T}$ are all independent).
\end{itemize}
Reversely, the existence of a tree of random variables $\{(X_t,\xi_{t})\}_{t\in\T}$ satisfying the properties above implies that the measure~$\mu$ is stationary (this is an easy consequence of equality \eqref{vertex} and independence).

\paragraph{The cut-off operators~--} For a fixed RDE associated with a relation function $R$ and a measure $m$, we introduce a family of operators $\Phi_a$, $a\in\R$, as follows: considering the new relation function
\begin{equation}\label{eq:Ra}
R_{a}(X_1,\dots,X_d;\xi):=\min(R(X_1,\dots,X_d;\xi),a),
\end{equation}
the \emph{cut-off operator} $\Phi_a$ is defined as in \eqref{eq:RDE}. That is, the operator $\Phi_a$ sends any measure $\mu\in\cP$ to the law of $R_{a}(X_1,\dots,X_d;\xi)$, where all the $X_i$'s have law $\mu$, $\xi$ has law $m$, and they are all independent.

\paragraph{Assumptions and some remarks~--}\label{s:assumptions}
In this paragraph we list the assumptions for Theorem~\ref{t:main}. Most likely, these may not be the most general assumptions that make our strategy work, but they are general enough to treat the particular problems of Theorems~\ref{t:applying} and~\ref{t:TSP}.

\begin{assumption}\label{a:incr}
The function $R$ is monotone non-decreasing in each of the first $d$ variables $X_1,\dots,X_d$. 
\end{assumption}

\begin{rem}
The monotonicity of $R$ implies that~$\Phi$ preserves the 
\emph{stochastic order}~$\preceq$ (recall that $\mu\preceq \mu'$ if there is a coupling $(X,X')$ between these measures 
such that $X\le X'$ almost surely).
\end{rem}

\begin{assumption}\label{a:translation}
The function $R$ commutes with the translations: for any $c\in\R$ and any $x_1,\dots,x_d,y\in \R$
\begin{align*}
R(x_1+c,\ldots,x_d+c;y) &= R(x_1,\ldots,x_d;y)+c.
\end{align*}
\end{assumption}

\begin{rem}
Some natural processes (for instance, the one for random metrics) are rather scale-equivariant, that is, they satisfy the relation
\[
R(\lambda z_1,\ldots,\lambda z_d;y) = \lambda R(z_1,\ldots,z_d;y).
\]
For these processes, we pass to the logarithmic scale, that is, we rather consider~$X_t:=\log Z_t$.
\end{rem}

\begin{assumption}\label{a:space}
There exists a subset $\cQ$ of $\cP$ and a metric $d_\cQ$ defined on $\cQ$ which induces a topology (non-strictly) finer than the weak-$*$ topology. The push-forward of the action of $\R$ by translations on~$[-\infty,+\infty]$ to the space of measures $\cP$ preserves~$\cQ$ and restricts to a continuous action on $(\cQ,d_\cQ)$. 

We also require that $\cQ$ contains the Dirac measures $\{\ddelta_a\}_{ a\in \R}$, and that the operators $\Phi$ and $\Phi_a$'s, $a\in\R$, preserve the space $\cQ$, and are continuous on this space (with respect to the the topology defined by the distance~$d_{\cQ}$).
\end{assumption}

From now on, we presume that such a subset $\cQ$ and the distance $d_{\cQ}$ are chosen and fixed.

\begin{assumption}\label{a:stationary}
There exists a stationary measure $\bm\in \cP_0 \cap \cQ$.
\end{assumption}

Again, consider $\bm$ to be chosen and fixed. For any $c\in\R$ we set $\bm_c$ to be the translation of $\bm$ by~$c$: if $X$ has law $\bm$, then $X+c$ has law $\bm_c$. 
\begin{rem}
The assumption~\ref{a:translation} of translation-equivariance implies that all the translates $\bm_c$ of the measure~$\bm$ are also stationary. In particular, this implies that even if such a random tree process is endogenous, in order to define the $X_t$'s one has to ``fix a scale'' by choosing one of the translates of~$\bm$: replacing it with $\bm_c$ will add $c$ to all the $X_t$'s.
\end{rem}

\begin{assumption}\label{a:everywhere}
The space $\cQ$ is contained in the basin of attraction of the set of translates of $\bm$. More precisely, for any $\mu\in \cQ$ there exists a \emph{center} $c=\tc(\mu)\in\R$ such that 
\begin{equation}\label{eq:center}
\Phi^n[\mu] \xrightarrow[]{d_{\cQ}} \bm_{\tc(\mu)} \quad\text{as } n\to\infty.
\end{equation}
\end{assumption}

\begin{assumption}\label{a:infinity}\label{a:convergence} 
There exists a compact set $\cQ_0\subset \cQ$, having $\bm$ as an interior point, such that the \emph{center function} $\tc:\cQ\to \R$ is defined (by~\eqref{eq:center}) and continuous on~$\cQ_0$, and that the convergence in~\eqref{eq:center} is uniform on~$\cQ_0$.
\end{assumption}

\begin{rem}\label{r:Dirac}
The equivariance assumption~\ref{a:translation} implies that we have an analogous convergence to the translates of $\bm$ in a neighbourhood of any Dirac measure $\ddelta_a$, $a\in\R$, and that we have $\tc(\ddelta_a)\to+\infty$ as $a\to +\infty$. Also, the function~$\tc$ is in fact continuous of $\cQ$, as any initial measure~$\mu$ after a finite number of (continuous) iterations of $\Phi$ falls in a $\tc(\mu)$-translate of $\cQ_0$, where the continuity of the center function is guaranteed by~\ref{a:convergence}. For the same reason, the convergence~\eqref{eq:center} is uniform on compact subsets of~$\cQ$, wherever it takes place pointwise (in particular, on compact subsets of~$\cQ$).
\end{rem}

\begin{rem}
In fact, as the reader will see all along the proofs, the assumption~\ref{a:everywhere} can be dropped: the assumption~\ref{a:convergence} suffices. Nevertheless, we prefer to impose it for the clarity of the exposition (otherwise one has to introduce further notations and to keep track of the domains of definition, etc.).
\end{rem}

The continuity of the center function $\tc$ implies that a cut-off of the stationary measure $\bm$
for a large value $a$ has quite a small effect on the center of the measure. Our last assumption quantifies this effect:
\begin{assumption}\label{a:superexp} 
Define the function $\Delta:\R\to \R$ by setting $\Delta(a):=-\tilde{c}(\Phi_a[\bm])$. Then $\Delta$ tends to zero superexponentially as $a\to+\infty$, that is, for any~$\beta>0$
\[
\Delta(a)= o\left (e^{-\beta a}\right ).
\]
\end{assumption}

\begin{rem}
The asymmetry in the assumption~\ref{a:superexp} is due to the fact that we are using the \emph{upper} cut-off; some functions $R$ may require the \emph{lower} cut-off and thus the corresponding reversion of assumptions and arguments.
\end{rem}

\section{The cut-off method revised}\label{s:cut-off}

\paragraph{Earlier appearances --}
We begin this part by recalling some generalities about the cut-off method. It was introduced in~\cite[\S 4.3]{KKT} in order to find solutions for some translation-equivariant functions~$R$ (corresponding to random metrics on hierarchical graphs defined by multiplicative cascades). A main obstacle was the drift: in fact, one had to consider a family of RDEs
\begin{equation}\label{eq:s}
X\stackrel{d}{=}R(X_1,\ldots,X_d;\xi)+s,
\end{equation}
parametrized by the (also unknown!) \emph{drift constant}~$s$. Among these parameters, there is at most one $s=s_{cr}$ for which \eqref{eq:s} possesses a ``physically reasonable'' solution: for $s<s_{cr}$ the iterations of the associated $\Phi$ make every initial measure converge to $-\infty$, whereas for $s>s_{cr}$, the convergence is to $+\infty$. 

To avoid this problem, the relation function $R$ is replaced by the cut-off function $R_a$ defined by~\eqref{eq:Ra}, where a cut-off is applied, and consequentially $\Phi$ is replaced by $\Phi_a$.
Then the transformation $\Phi_a$ sends any measure to a measure supported on $[-\infty,a]$, and one can find a stationary measure for the new process as the (stochastically monotone) limit of iterations $\Phi_a^n[\ddelta_{+\infty}]$ (cf.~\cite[Lemma~4]{RDE}). This limit is non-trivial (that is, not concentrated at $-\infty$) if and only if $s>s_{cr}$. One then finds a stationary measure for the initial process (so $s=s_{cr}$) by taking a non-trivial ``diagonal'' limit as $s\searrow s_{cr}$ and $a\to +\infty$ (the latter ensures that the effect of the cut-off disappears).

\paragraph{Cut-off method --}
Similar arguments of monotonicity can be applied to ensure the endogeny of the initial process. The key variation comes from noticing that the cut-offs that are made in the past may be done at different values for different times. 

First, for the simplicity of the exposition, we will impose additional assumptions. Namely, assume that
\begin{itemize}
\item the dimension $d$ is finite;
\item the relation function $R$ is continuous in $X_1,\dots,X_d$;
\item the relation function takes finite values whenever all its arguments are finite.
\end{itemize}
All these assumptions hold in the setting of random metrics (Section~\ref{s:metrics}) and the exposition under these assumptions 
clarifies the main ideas. However, they are unnecessary for establishing our main theorem, and we will remove them at the end of this section; see Proposition~\ref{p:removing}.

\smallskip

Consider any sequence $\bA=\{a_n\}_{n=1}^{\infty}$ of cut-off values $a_n\in(-\infty,+\infty]$, and the corresponding sequence of relation functions $R_{a_n}$ (where $R_{+\infty}$ is nothing else but the initial recursion relation~$R$). We are going to construct (cf.~\cite[\S\,4.3]{KKT}) a random tree process $\{(X^\bA_t,\xi_t)\}_{t\in\T}$, whose variables verify
\begin{equation}\label{eq:R-a}
X^{\bA}_t=R_{a_{\|t\|}}(X^{\bA}_{t_1},\ldots,X^{\bA}_{t_d};\xi_t)
\end{equation}
instead of \eqref{vertex}; observe that the recursion relation changes with the distance to the root.

\begin{dfn}
Let a sequence $\bA=\{a_n\}_{n=1}^{\infty}$, $a_n\in(-\infty,+\infty]$, be given. For any family of values $\{\xi_t\}_{t\in\T}$ and any $k\in\N$, we recursively define $\{X_t^{\bA,k}\}_{t\in\T}$ as
\begin{equation}\label{RTP_cutoff}
\begin{cases}
X_t^{\bA,k}=+\infty & \text{if }\|t\|\ge k,\\
X^{\bA,k}_t=R_{a_{\|t\|}}(X^{\bA,k}_{t_1},\ldots,X^{\bA,k}_{t_d};\xi_t)& \text{if }\|t\|<k.
\end{cases}
\end{equation}
\end{dfn}

Directly from this definition, assumption~\ref{a:incr}, and using the assumed continuity of the function~$R$, we have:
\begin{prop}\label{p:limit_cutoff}
Let $\bA=\{a_n\}_{n=1}^{\infty}$, $a_n\in(-\infty,+\infty]$, be a sequence and let $\{\xi_t\}_{t\in\T}$ be a family of real numbers. For any fixed $t\in\T$, the sequence $\{X_t^{\bA,k}\}_{k=1}^{\infty}$ is non-increasing. In particular, the limit
\begin{equation}\label{e:limit_decreas}
X^{\bA}_t:=\lim_{k\to\infty}X_t^{\bA,k}
\end{equation}
exists (though it may be infinite) and satisfies the relation~\eqref{eq:R-a}.

When some $a_n$ is finite, then for any $t\in\T$ such that $\|t\|\le n$, the limit $X_t^\bA$ is not $+\infty$. In particular, if there are infinitely many finite $a_n$'s, all the limits $X_t^\bA$ are not $+\infty$ (though they still can be~$-\infty$).

Moreover, for a fixed family $\{\xi_t\}_{t\in\T}$, this limit is monotone on $\bA$: if for $\bA=\{a_n\}$ and $\bA'=\{a'_n\}$ we have $a_n\le a'_n$ coordinate-wise then for any $t\in\T$, we have $X_t^{\bA}\le X_t^{\bA'}$.
\end{prop}

When the family $\{\xi_t\}_{t\in\T}$ is considered as a family of random variables (i.i.d.~of law $m$), the definition \eqref{RTP_cutoff} makes the $X_t^{\bA,k}$'s random variables. By the symmetry in the construction, all the laws of the $X_t^{\bA,k}$'s coincide
at any fixed depth $\|t\|=n$; we denote this law by $\mu_{n}^{\bA,k}$. It follows from~\eqref{RTP_cutoff} that these laws satisfy (the non-stationary) RDEs:
\begin{equation}\label{measures_cutoff}
\begin{cases}
\mu_{k}^{\bA,k}=\ddelta_{+\infty},\\
\mu_n^{\bA,k}=\Phi_{a_{n}}[\mu^{\bA,k}_{n+1}]& \text{if }n<k.
\end{cases}
\end{equation}
Passing to the limit in \eqref{measures_cutoff}, as $k\to\infty$, we get that the laws $\mu_{\|t\|}^\bA$'s of the  $X_t^{\bA}$'s from Proposition~\ref{p:limit_cutoff} also satisfy the RDEs
\begin{equation}\label{e:RDE-bA}
\mu_n^{\bA}=\Phi_{a_{n}}[\mu^{\bA}_{n+1}].
\end{equation}
By construction, the RTP $\{(X_t^{\bA},\xi_t)\}_{t\in\T}$ is \emph{endogenous} (cf.~\cite[Lemma~15]{RDE}).

\paragraph{Removing the cut-offs --}

The next step is to start \emph{removing} the cut-offs. The non-stationary process, constructed with Proposition~\ref{p:limit_cutoff}, is endogenous, but the cut-offs alter the recurrence relation and the distribution of the~$X_t$'s. The second step is to start removing the cut-offs one by one, starting from above, and thus constructing a new family of endogenous RTPs (this time an increasing one). Explicitly:
\begin{dfn}
Let a sequence $\bA=\{a_n\}_{n=1}^{\infty}$, $a_n\in(-\infty,+\infty]$,  be given. Define a family of new sequences $\bA^{(\ell)}=\{a_n^{(\ell)}\}$ by
\[
a_n^{(\ell)}=\begin{cases}
a_n,& \text{if }n\ge \ell,\\
+\infty, &\text{if }n<\ell.
\end{cases}
\]
\end{dfn}
The RDEs \eqref{e:RDE-bA} for the new sequences are
\begin{equation}\label{eq:RDE-bA-l}
\mu_n^{\bA^{(\ell)}}=\begin{cases}
\Phi_{a_n}[\mu_{n+1}^{\bA^{(\ell)}}],& \text{if }n\ge \ell,\\
\Phi[\mu_{n+1}^{\bA^{(\ell)}}],
&\text{if }n<\ell.
\end{cases}
\end{equation}
Notice that when $n<\ell$, one has
\begin{equation}\label{eq:RDE-bA-l2}
\mu_n^{\bA^{(\ell)}}=\Phi^{\ell-n}[\mu_\ell^{\bA^{(\ell)}}].
\end{equation}
The following proposition is deduced from the monotonicity statement in Proposition~\ref{p:limit_cutoff}, and again from the continuity assumption for the function~$R$:
\begin{prop}
Let $\bA=\{a_n\}_{n=1}^{\infty}$, $a_n\in(-\infty,+\infty]$, 
be a sequence 
and let $\{\xi_t\}_{t\in\T}$ be a family of real numbers. For any fixed $t\in\T$, the sequence $\{X_t^{\bA^{(\ell)}}\}_{\ell=1}^{\infty}$ is non-decreasing. In particular the limit
\begin{equation}\label{eq:limit_removed}
\tX^{\bA}_t:=\lim_{\ell\to\infty}X_t^{\bA^{(\ell)}}
\end{equation}
exists, and satisfies the recurrence relation
\begin{equation}\label{eq:tilde-R-a}
\tX^{\bA}_t=R(\tX^{\bA}_{t_1},\ldots,\tX^{\bA}_{t_d};\xi_t)
\end{equation}
\end{prop}
Observe that the limits \eqref{eq:limit_removed} can possibly be infinite, even if we assume that there are infinitely many finite $a_n$'s.

As above, when the family $\{\xi_t\}_{t\in\T}$ is considered as a family of random variables (i.i.d.~of law $m$), the definition \eqref{eq:limit_removed} makes the $\tX_t^{\bA}$'s random variables. For any $k\in\N$, the limit
\begin{equation}\label{eq:limit_removed_meas}
\tmu_{k}^{\bA}=\lim_{\ell\to\infty}\mu_k^{\bA^{(\ell)}}
\end{equation}
is the law of the $\tX_t^{\bA}$'s with $\|t\|=k$. They satisfy the RDEs
\[
\tmu^{\bA}_k=\Phi[\tmu^{\bA}_{k+1}],
\]
though they can be trivial solutions (namely~$\tmu^{\bA}_k=\ddelta_{+\infty}$ or~$\tmu^{\bA}_k=\ddelta_{-\infty}$).

Nonetheless, by construction, for any sequence $\bA$, \emph{the process $\{\tX^{\bA}_t\}_{t\in\T}$ is endogenous}. Hence, Theorem~\ref{t:main} will be proved -- under our additional assumptions -- as soon as we establish the following:
\begin{prop}\label{p:main}
Let $\bm$ be a solution to a RDE
\begin{equation}\label{eq:RDE_simple}
\bm=\Phi[\bm],
\end{equation}
given by the assumption~\ref{a:stationary}, where the recurrence relation $R$ and the measure $m$ satisfy all the assumptions~\ref{a:incr}--\ref{a:superexp}.
Then there exists a sequence~$\bA$ such that $\tmu^{\bA}_n=\bm$ for any~$n$. 

Moreover, as a technical conclusion, for the constructed sequence~$\bA$ we have $\mu^{\bA^{(\ell)}}_n\in \cQ$ for any~$n$ and~$\ell$.
\end{prop}

Roughly speaking, to prove Proposition~\ref{p:main}, we have to choose a sequence $\bA=\{a_n\}$ of cut-offs such that any tail $\{a_n\}_{n=\ell+1}^{\infty}$ suffices to bring initially-infinite values to the core, while the initial part $ \{a_n\}_{n=1}^{\ell}$ should not push it too far towards $-\infty$.

Before getting into the proof of Proposition~\ref{p:main}, we conclude this paragraph by explaining how to get rid of the continuity assumptions (that could be a problem when there is an infinite number of variables).  In particular the reader will see how the last conclusion of Proposition~\ref{p:main} is employed (this is the only step where we need it).
 Namely, we have the following:
\begin{prop}\label{p:removing}~Under the assumptions~\ref{a:incr}--\ref{a:superexp} the following statements hold:
\begin{itemize}
\item Assume that the sequence $\bA$ is such that all the measures $\mu_n^{\bA}$'s belong to the domain~$\cQ$. Then the relation~\eqref{eq:R-a} is satisfied almost surely.
\item Assume that the sequence $\bA$ is such that all the measures $\mu_n^{\bA^{(\ell)}}$, $\tilde{\mu}_n^{\bA}$'s  belong to the domain~$\cQ$. Then the relation~\eqref{eq:tilde-R-a} is satisfied almost surely.
\end{itemize}
\end{prop}
\begin{proof}
For any given set of values of $\{\xi_t\}$ the sequences $\{X_t^{\bA,k}\}$ are monotonous non-increasing. Hence, passing to the limit in~\eqref{RTP_cutoff} and using the monotonicity of the function $R$, we get 
\begin{multline}\label{eq:m-lim}
X_t^{\bA} = \lim_{k\to\infty} X_t^{\bA,k} = \lim_{k\to\infty} R_{a_{\|t\|}}(X^{\bA,k}_{t_1},\ldots,X^{\bA,k}_{t_d};\xi_t) \ge \\ R_{a_{\|t\|}} (\lim_{k\to \infty} X^{\bA,k}_{t_1},\ldots, \lim_{k\to \infty} X^{\bA,k}_{t_d};\xi_t)  = R_{a_{\|t\|}}( X_{t_1}^{\bA},\dots, X_{t_d}^{\bA}; \xi_t).
\end{multline}
Whence $X_t^{\bA}\ge R( X_{t_1}^{\bA},\dots, X_{t_d}^{\bA}; \xi_t)$. 

Now, recall that $\xi_t$ are in fact random variables. The law of the left hand side of~\eqref{eq:m-lim} is $\mu_{\|t\|}^{\bA}$, the law of the right hand side is by definition $\Phi[\mu_{\|t\|+1}^{\bA}]$, and \eqref{eq:m-lim} gives the stochastic comparison
\begin{equation}\label{eq:mu-stochastic}
\Phi_{a_{\|t\|}}[\mu_{\|t\|+1}^{\bA}] \preceq \mu_{\|t\|}^{\bA}.
\end{equation}

Now, showing that the inequality in~\eqref{eq:m-lim} is in fact \emph{almost surely} an equality is equivalent to prove the equality in~\eqref{eq:mu-stochastic}. Let $n=\|t\|$; by assumption $\mu_n^{\bA}\in\cQ$, and the operator $\Phi_{a_n}$ is continuous at~$\mu_{n+1}^{\bA}$. As $\mu_{n}^{\bA,k}\to \mu_n^{\bA}$ as $k\to\infty$, and we have $\mu_{n}^{\bA,k}=\Phi_{a_n}[\mu_{n+1}^{\bA,k}]$ for any $k> n$, passing to the limit we get the desired equality $\mu_n^{\bA}=\Phi_{a_n}[\mu_{n+1}^{\bA}]$. 

The proof of the second part of the proposition goes in the same way, with the only difference that the sequence $X_n^{\bA^{(\ell)}}$ is now increasing, so that the inequalities are reversed.
\end{proof}

Proposition~\ref{p:removing} shows that even without the additional assumptions, the family of random variables $\{\tX_{t}\}_{t\in\T}$, constructed using the sequence~$\bA$ from Proposition~\ref{p:main}, almost surely satisfies the recurrence relation~\eqref{eq:tilde-R-a}. This, together with the endogeny of the family $\tX_t$, implies Theorem~\ref{t:main}.

\paragraph{Blocks of cut-offs --} Recall from assumption~\ref{a:space} that we have a metric~$d_{\cQ}$ on the space~$\cQ$ of probability measures; all distances are considered with respect to this metric.
The construction of the desired sequence $\bA$ will be done by blocks:
\begin{dfn}
A \emph{block} $\bB$ is a finite sequence $b_1,\ldots,b_n\in (-\infty,+\infty]$. Associated with a block, we define the \emph{block operator} $\Phi_\bB:\cP\to\cP$ which is the composition of the cut-off operators
\[
\Phi_\bB=\Phi_{b_1}\circ\cdots \circ\Phi_{b_n}.
\]
The integer $n$ is called the \emph{size} of the block $\bB$ and is denoted by $|\bB|$.
\end{dfn}

The blocks that we choose should work together, so let us quantify their properties. Namely, 
for each block we will specify how it handles ``as an input'', on the one hand, the  Dirac measure at 
infinity (the result should be close to $\bm$), and on the other hand, a measure close to $\bm$.
To do so, we introduce the following:

\begin{dfn}\label{d:block}
Given $\delta>\delta'>0$, $\eps,\eps'>0$, we say that a block $\bB$ is a \emph{$(\delta,\delta';\eps,\eps')$-block} if
\begin{itemize}
\item the image $\Phi_{\bB}[\ddelta_{+\infty}]$ is $\eps$-close to $\bm$,
\item for any $c'\in(-\delta',0)$ and any $\mu\in\cQ$ which is $\eps'$-close to $\bm_{c'}$, there exists $c\in (-\delta,0)$ such that the image $\Phi_{\bB}[\mu]$ is $\eps$-close to the measure $\bm_c$ (see Fig.~\ref{f:block}).
\end{itemize}
\end{dfn}

\begin{figure}
\[\includegraphics[scale=1]{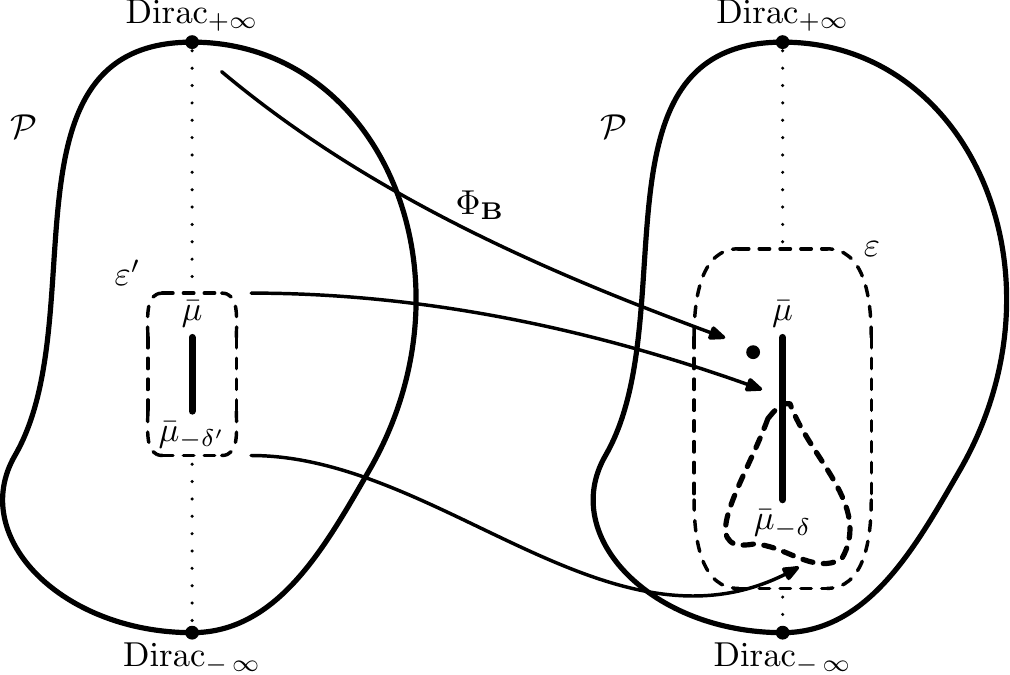}\]
\caption{Properties of a $(\delta,\delta';\eps,\eps')$-block $\bB$.}\label{f:block}
\end{figure}

Observe that for arbitrary choices of $\delta,\delta',\eps,\eps'$, $(\delta,\delta';\eps,\eps')$-blocks may not exist (for instance, one should expect non-existence when~$\eps\ll \eps'$). On the other hand, any $(\delta,\delta';\eps,\eps')$-block is automatically a $(\delta,\delta';\eps,\eps'')$-block for any $\eps''<\eps'$.

As we shall explain in the next paragraph, Proposition \ref{p:main} is a consequence of the following:
\begin{prop}\label{p:blocks}
Under the assumptions~\ref{a:incr}--\ref{a:superexp} the following statement holds: for any $\delta>\delta'>0$, $\eps>0$, there exist $\eps'>0$ and a $(\delta,\delta';\eps,\eps')$-block.
\end{prop}

\paragraph{Proposition \ref{p:blocks} implies Proposition \ref{p:main} --} 
Take $\eps_1$ sufficiently small so that the $2\eps_1$-neighbourhood of $\bm$ belongs to~$\cQ_0$ and
 choose $\delta_0 > 0$ so that for any $c\in(-\delta_0,0)$ we have $d_\cQ(\bm_c,\bm)<\eps_1$. 
Consider any sequence $\delta_n\in(0,\delta_0)$ that decreases to~$0$ (for instance $\delta_n=2^{-n}\delta_0$) 
and define recursively a sequence $\eps_n$ in such a way that $\eps_n\le \tfrac{\eps_1}{n}$ and such that there exists a $(\delta_{n-1},\delta_n;\eps_{n-1},\eps_n)$-block $\bB_n$. Let us check that the sequence~$\bA$, obtained by juxtaposing the blocks $\bB_n$, satisfies the conditions of Proposition \ref{p:main}. To fix notations, we write $\bA=\bB_1\bB_2\cdots$ and $s_n=|\bB_1|+\ldots+|\bB_n|$.

As both limits \eqref{e:limit_decreas} and \eqref{eq:limit_removed} exist (possibly infinite), we can safely define the $X^{\bA}_t$'s and the $\tX^{\bA}_t$'s by considering only the limits along subsequences. We use this remark to be able to apply blocks of cut-offs entirely.

\begin{lem}\label{l:blocks_implies}
For any $n<k$ there exists $c_{n,k}\in (-\delta_n,0)$ such that $\mu_{s_n}^{\bA,s_k}$ is $\eps_n$-close to $\bm_{c_{n,k}}$.
\end{lem}
\begin{proof}
By \eqref{measures_cutoff}, we have
\[
\mu_{s_{n-1}}^{\bA,s_k}=\Phi_{\bB_n}\cdots\Phi_{\bB_k}[\ddelta_{+\infty}].
\]
The measure $\mu_{s_k}^{\bA,s_{k-1}}=\Phi_{\bB_k}[\ddelta_{+\infty}]$ is $\eps_{k-1}$-close to $\bm$, since $\bB_k$ is a $(\delta_{k-1},\delta_k;\eps_{k-1},\eps_k)$-block.

Therefore we can prove the statement by (backwards) induction on $n$, taking $n=k-1$ as the base of the induction, and with the induction step provided by the second part of Definition~\ref{d:block}.
\end{proof}

Passing to the limit as $k\to\infty$ in the conclusion of Lemma~\ref{l:blocks_implies}, we see that $\mu_{s_n}^{\bA}$ is $\eps_n$-close to~$\bm_{c_{n}}$, for some $c_n\in [-\delta_n,0)$. Now, considering $\mu_{s_n}^{\bA^{(s_\ell)}}$ and recalling the RDEs \eqref{eq:RDE-bA-l}, we get the same conclusion for $n\ge \ell$, while for $n<\ell$ we get (see~\eqref{eq:RDE-bA-l2})
\[
\mu_{s_n}^{\bA^{(s_\ell)}}=\Phi^{s_\ell-s_n}[\mu_{s_\ell}^{\bA^{(s_\ell)}}].
\]
The measure $\mu_{s_\ell}^{\bA^{(s_\ell)}}$ is $\eps_\ell$-close to $\bm_{c_\ell}$, with $c_\ell\in(-\delta_0,0)$ and $\eps_\ell\le \eps_1$; after the choices for $\eps_1$ and $\delta_0$, the triangular inequality guarantees that the measure $\mu_{s_\ell}^{\bA^{(s_\ell)}}$ belongs to $\cQ_0$:
\[
d_{\cQ}(\mu_{s_\ell}^{\bA^{(s_\ell)}},\bm)\le d_{\cQ}(\mu_{s_\ell}^{\bA^{(s_\ell)}},\bm_{c_\ell})+d_{\cQ}(\bm_{c_\ell},\bm)< 2\eps_1.
\]
Now, recall that due to the assumption~\ref{a:convergence} the powers $\Phi^k$ converge uniformly on $\cQ_0$; also, the measures~$\mu_{s_{\ell}}^{\bA^{(s_{\ell})}}$ converge to $\bm$ as $\ell\to\infty$. For a uniformly convergent family of continuous mappings, the limit of the images of a convergent family of points can be calculated by substituting the limit point to the limit (continuous) map.
Therefore for any $n$ one has
\[
\lim_{\ell\to\infty} \mu_{s_n}^{\bA^{(s_\ell)}}= \lim_{\ell\to\infty} \Phi^{s_\ell-s_n}[\mu_{s_\ell}^{\bA^{(s_\ell)}}]=\bm.
\]
Hence, recalling~\eqref{eq:limit_removed_meas}, we have the equality $\tmu_{s_n}^{\bA}=\bm$ for any $n$, and thus $\tmu_{n}^{\bA}=\bm$ for any $n$. 

\smallskip 

Finally, let us verify the technical assumption in Proposition~\ref{p:main}. Namely, we get $\mu_{s_n}^{\bA^{(s_\ell)},s_k}\in\cQ_0$ for any $k>n>\ell$; due 
to the compactness of $\cQ_0$ we have $\mu_{s_n}^{\bA^{(s_\ell)}}=\lim_{k\to\infty} \mu_{s_n}^{\bA^{(s_\ell)},s_k}\in \cQ_0$. Finally, the invariance of $\cQ$ under all the $\Phi_a$'s implies that $\mu_n^{\bA^{(\ell)}}\in \cQ$ for every $n$ and $\ell$.

This ends the proof of Proposition~\ref{p:main}, assuming Proposition~\ref{p:blocks}.

\section{Proposition~\ref{p:blocks}: constructing the blocks}\label{s:construction}
\paragraph{Subblocks --} The most delicate part of this work is to construct the desired blocks.
The blocks we are looking for will be composed of elementary subblocks:
\begin{dfn}
A \emph{subblock} $\bb$ is a block such that only the last value $b_{|\bb|}$ is finite: for every $1\le k\le |\bb|-1$, $b_k=\infty$.
We denote by $\bb(a,\ell)$ the subblock defined by $b_{|\bb|}=a$, $|\bb|=\ell+1$. So we have that the subblock operator
\[
\Phi_{\bb(a,\ell)}=\Phi^\ell\circ \Phi_a
\]
is the composition of one cut-off operator and $\ell$ iterations of the operator $\Phi$.
\end{dfn}

\paragraph{Long subblocks and centers --}

The convergence assumption \ref{a:everywhere} ensures that for a given initial measure $\mu\in\cQ$,
 if we take a very long subblock $\bb(a,\ell)$ (that is, with $\ell$ sufficiently large), then the image $\Phi_{\bb(a,\ell)}[\mu]$ will be close to $\bm_{\tc(\Phi_a[\mu])}$, and actually the former converges to the latter as~$\ell$ goes to~$\infty$.

When sufficiently long subblocks are applied consecutively, we would like to describe the joint result of their application. The application of each of them leaves us with a measure close to some~$\bm_c$. This leads to the study the dynamics of the centers $c$:
by the continuity of the function~$\tc$, applying a sufficiently long subblock $\bb(a,\ell)$ to a measure $\mu$ close to a stationary measure~$\bm_c$, we obtain a measure 
that is close to the stationary measure $\bm_{\tc(\Phi_a[\bm_c])}$. This motivates the following:

\begin{dfn}[Center function]
For any $a\in\R$, we define the function $S_a:\R\to\R$ by
\[
S_a(c)=\tc(\Phi_a[\bm_c]),
\]
\end{dfn}
In other words, the function $S_a$ locates the center of the measure that we obtain by applying long subblock operators $\Phi_{\bb(a,\ell)}$ to $\bm_c$.

This definition naturally extends by continuity to $c=+\infty$, by putting $S_{a}(+\infty)=\tc(\ddelta_a)$.
For a given $a\in\R$, considering $S_a$ as a self-map of $\R$, we write $S_a^n$ to denote the $n$-fold composition $S_a\circ \cdots \circ S_a$.

\begin{rem}\label{r:S-increases}
As a consequence of the assumption \ref{a:incr} that the operator $\Phi$ preserves the stochastic order, $S_a(c)$ is an increasing function of the cut-off value $a$. Also by the continuity of the center function (assumption~\ref{a:convergence}) the center $S_a(c)$ tends to $c$ as $a$ goes to $+\infty$.
\end{rem}
Sometimes it will be easier to work with the \emph{displacement of centers}, rather than with their new location:
\begin{dfn}[Displacement function]
For any $a\in\R$, we define \emph{the displacement function} $\Delta_a:\R\to(0,+\infty)$ by
\begin{equation}\label{e:Delta}
\Delta_a(c)=c-S_a(c).
\end{equation}
\end{dfn}

Notice that the displacement function takes positive values only because of Remark~\ref{r:S-increases}. Moreover, the same remark implies that $\Delta_a(c)$ is a decreasing function of $a$. We can be more precise:
\begin{lem}\label{l:delta}
For any $a$, $c\in\R$, we have $\Delta_a(c)=\Delta_{a-c}(0)$. In particular $\Delta_a(c)$ is increasing with respect to $c$ and decreasing with respect to $a$.
\end{lem}
\begin{proof}
The first statement follows from the translation-equivariance~\ref{a:translation}. The second one is a corollary of the first one and of Remark~\ref{r:S-increases}.
\end{proof} 
Observe that
\[
\Delta_a(0)=-S_a(0)=-\tc(\Phi_a[\bm]),
\]
therefore, as function of $a$,  $\Delta_a(0)$ is exactly the function $\Delta$ defined in assumption \ref{a:superexp}. From the previous lemma, we have $\Delta_a(c)=\Delta(a-c)$.

\paragraph{Sketch of construction --}\label{par:sketch} The desired $(\delta,\delta';\eps,\eps')$-block $\bB$ will be defined by a juxtaposition of subblocks constructed in the following way. 
First, we apply a subblock $\bb_0=\bb(b_0,\ell_0)$ with the following properties:
\begin{enumerate}
\item the cut-off value $b_0$ is chosen sufficiently large so that it almost does not affect the measures $\bm_c$'s, with $c\in [-\delta',0]$;
\item the number of iterations $\ell_0$ is chosen sufficiently large, so that, after cutting off the Dirac measure $\ddelta_{+\infty}$ at $b_0$, $\Phi_{\bb_0}$ almost sends it to $\bm_{c_0}$, with $c_0=\tc(\ddelta_{b_0})=S_{b_0}(+\infty)$. 
\end{enumerate}

After this, we choose a second subblock $\bb_1=\bb(b_1,\ell_1)$ that we repeat (identically) sufficiently many times $N$, so that the corresponding transformation of $\cP$ will bring $\bm_{c_0}$ sufficiently close to $\bm$, and its ``side effect'', when applying it to $\bm_c$, with $c\in (-\delta',0)$, stays under control (these measures will not be pushed to the left farther than~$\bm_{-\delta}$). Making all of this work will provide us with the desired block
\[
\bB=\underbrace{\bb_1\cdots \bb_1}_{N\text{ times}}\bb_0;
\]
see Fig.~\ref{f:S} for how its application should treat the centers of measures.

\paragraph{Displacement of the centers --}

We shall start by choosing the cut-off values $b_0$ and $b_1$ for the subblocks and the number $N$ of repetitions of $\bb_1$, using the approximation by the dynamics of the centers. Then, in the next paragraph, we will choose the appropriate sizes~$\ell_0$, $\ell_1$, concluding the proof of Proposition~\ref{p:blocks}.

First, choose and fix $\delta''$ such that $\delta>\delta''>\delta'$. Note that as $b\to+\infty$, we have 
\[
S_b(-\delta') \nearrow -\delta'\]
by Remark~\ref{r:S-increases}, and
\[
 \tc(\ddelta_b)\to +\infty,
\]
for $\tc(\ddelta_b)=b+\tc(\ddelta_0)$.

Thus, there exists a sufficiently large $b_0$ such that $S_{b_0}(-\delta')\in(-\delta'',-\delta')$ and $\tc(\ddelta_{b_0})>0$. We choose and fix such $b_0$, and, as above, we write 
\[c_0:=\tc(\ddelta_{b_0}).\]
As a result of these choices, once we also choose (at the very end, \S~\ref{p:length}) $\ell_0$ sufficiently large, we shall get that:
\begin{enumerate}
\item \label{cond:bb0-1} the measure $\Phi_{\bb(b_0,\ell_0)}[\ddelta_{+\infty}]=\Phi^{\ell_0}[\ddelta_{b_0}]$ is \emph{sufficiently close} to $\bm_{c_0}$;
\item \label{cond:bb0-2} for every $c\in(-\delta',0)$, the measure $\Phi_{\bb(b_0,\ell_0)}[\bm_c]$  is \emph{sufficiently close} to one of the measures $\bm_{c'}$'s, with $c'\in (-\delta'',0)$.
\end{enumerate}

Next, let us choose the second cut-off value $b_1$ and the number $N$ of repetitions for $\bb_1$. To do so, we choose and fix $\eps_0$ such that the measures $\bm_c$'s, with $|c|<\eps_0$ are $\eps/2$-close to $\bm$, with respect to our metric $d_\cQ$ on the space $\cQ$. We shall prove the following:
\begin{lem}\label{l:7}
There exist $b_1$ and $N$ such that 
\begin{enumerate}[1')]
\item \label{cond:S^N-1} $S^N_{b_1}(c_0)$ is $\eps_0$-close to zero, and
\item \label{cond:S^N-2} $S^{N}_{b_1}(-\delta'')\in (-\delta,-\delta'')$.
\end{enumerate}
\end{lem}
This is exactly what we need for constructing the block (Definition~\ref{d:block}): condition~\ref{cond:S^N-1}') ensures that the measure $\Phi_{\bB}[\ddelta_{+\infty}]=\Phi_{\bb_1}^N\Phi_{\bb_0}[\ddelta_{+\infty}]$ is close to~$\bm$, while condition~\ref{cond:S^N-2}') corresponds to the second one in Definition~\ref{d:block}.

\begin{figure}
\[\includegraphics[scale=1]{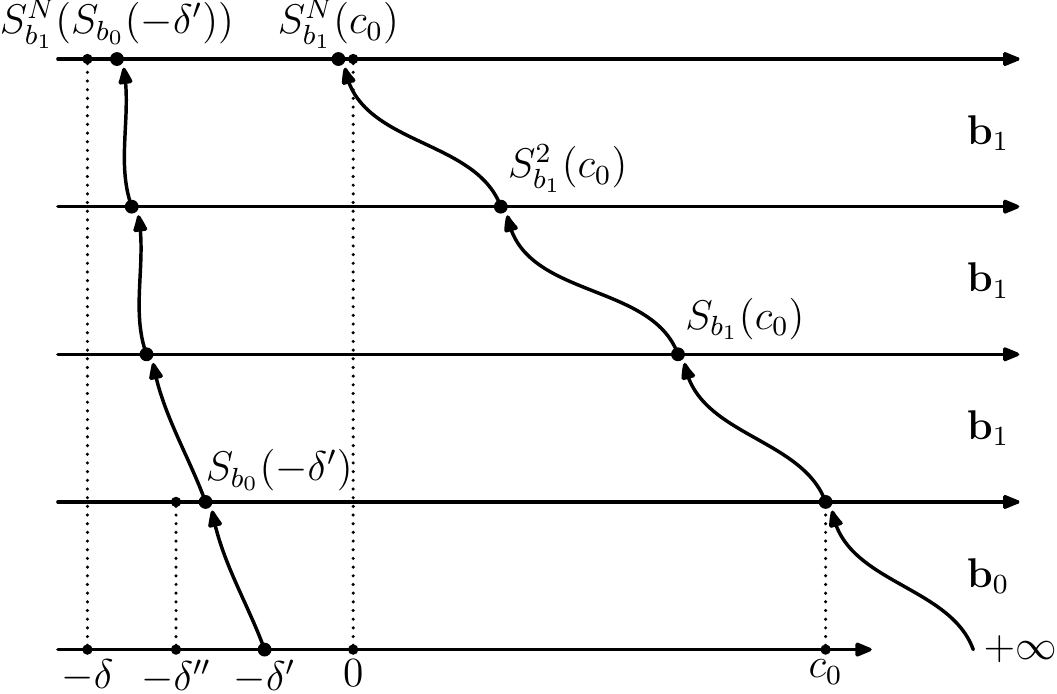}\]
\caption{Construction of the block $\bB$: choices for $b_1$ and $N$.}\label{f:S}
\end{figure}

\begin{proof}
For any $b$, the map $c\mapsto S_b(c)$ has no fixed points on the real line, whence 
\[S_b^n(c_0)\to -\infty\quad\text{as }n\to\infty.\]
In particular, for some number $n$ of iterations, we have $S_b^n(c_0)<0$. Take~$N$  to be the first such iteration (thus $N$ depends on $b$, not chosen yet!); that is, for any given $b\in \R$ define
\begin{equation}\label{e:N(b)}
N(b)=\min\{n\mid S_{b}^{n}(c_0)\le 0\}.
\end{equation}
With this choice, 
\begin{equation}\label{e:N(b)-in}
0\ge S_{b}^{N(b)}(c_0)=S_{b} (S_{b}^{N(b)-1}(c_0))>S_b(0),
\end{equation}
where in the last inequality we have used that the function $S_b$ is increasing. On the other hand, $S_b(0)\nearrow 0$ as $b\to+\infty$, hence there exists $b_1'$ such that $S_{b_1'}(0)\in(-\eps_0,0)$. Again by monotonicity of the center function, for any $b>b_1'$ we have 
\[
0\ge S_{b}^{N(b)}(c_0)> S_{b}(0) > S_{b_1'}(0) >-\eps_0,
\]
thus fulfilling  condition~\ref{cond:S^N-1}').

Now, we ought to choose $b_1>b_1'$ in such a way that  condition~\ref{cond:S^N-2}') is also fulfilled. 
To do so, note that for any $b$, $n$ and any initial center $c\in \R$, we have a telescopic sum:
\[
c-S_b^n(c)= \left(c-S_b^1(c)\right) +\dots + \left(S_b^{n-1}(c)-S_b^{n}(c)\right) = \sum_{j=0}^{n-1} \Delta_b(S_b^j (c)).
\]
After Lemma~\ref{l:delta}, the function $\Delta$ is monotone, thus we control the range of $c-S_b^n(c)$:
\begin{equation}\label{eq:sum-Delta}
n\Delta_b(c) > c-S_b^n(c) > n \Delta_b(S_b^{n-1} (c)). 
\end{equation}

The inequalities \eqref{e:N(b)-in} and \eqref{eq:sum-Delta} determine an upper bound for $N=N(b)$: by definition $S_b^N(c_0)\le 0 < S_b^{N-1}(c_0)$, and hence (provided $b\ge b_1'$)
\[
c_0+\eps_0> c_0-S_b^{N}(c_0) > N \Delta_b(S_b^{N-1} (c_0)) >  N \Delta_b(0).
\]
We then have 
\begin{equation}\label{eq:N-less}
N< \frac{c_0+\eps_0}{\Delta_b(0)}.
\end{equation}
From the first inequality in \eqref{eq:sum-Delta} we also have a first lower bound for $S^N_b(-\delta'')$: for any $b\in\R$,
\begin{equation}\label{eq:SbN}
S_b^N(-\delta'')>-\delta''-N\Delta_b(-\delta'').
\end{equation}
Joining the estimates in~\eqref{eq:N-less} and \eqref{eq:SbN}, for any $b\ge b_1'$ one has
\[
S_b^N(-\delta'') > -\delta'' - (c_0+\eps_0) \frac{\Delta_b(-\delta'')}{\Delta_b(0)}.
\]
To establish conclusion~\ref{cond:S^N-2}'), it suffices thus to find $b_1>b_1'$ such that 
\begin{equation}\label{eq:b1}
-\delta'' - (c_0+\eps_0) \frac{\Delta_{b_1}(-\delta'')}{\Delta_{b_1}(0)}>-\delta,
\end{equation}
or, in other words, such that 
\[
\frac{\Delta(b_1+\delta'')}{\Delta(b_1)}=\frac{\Delta_{b_1}(-\delta'')}{\Delta_{b_1}(0)}< \frac{\delta-\delta''}{c_0+\eps_0}.
\]
Under our assumption \ref{a:superexp}, the existence of such a ${b_1}$ is guaranteed. Indeed, set
\[
\frac{\delta-\delta''}{c_0+\eps_0}=\frac1C.
\]
Assuming the contrary, we would have that for any $a\ge b_1'$,
\begin{equation}\label{last_step}
\frac{\Delta(a+\delta'')}{\Delta(a)}\ge \frac1C;
\end{equation}
fixing any $a_0> b_1'$, we have the inequality \eqref{last_step} for every $a=a_0+k\delta''$, $k=0,1,\ldots$. Thus, for any~$m\in\N$, one has
\[
\frac{\Delta(a_0+m\delta'')}{\Delta(a_0)}=\prod_{k=0}^{m-1}\frac{\Delta(a_0+(k+1)\delta'')}{\Delta(a_0+k\delta'')}\ge \frac{1}{C^{m}}.
\]
So, for any $m\in\N$, one has $\Delta(a_0+m\delta'')\ge \Delta(a_0)\cdot C^{-m}$. However this contradicts the assumption~\ref{a:superexp}  of superexponential decrease of $\Delta$ (for $\beta=\frac{\log C}{\delta''}$).
This completes the proof of the lemma.
\end{proof}

\paragraph{Returning to the space of measures --}\label{p:length} The initial cut-off value $b_0$ has already been chosen. Now we choose and fix $b_1$ and $N$ satisfying the conclusions of Lemma~\ref{l:7}. We are ready to conclude the proof of Proposition~\ref{p:blocks}: we need to determine the lengths of the blocks $\bb_0$ and $\bb_1$ in order to have a good approximation of the dynamics with the center function.

We will need the following easy lemma, whose proof is left to the reader:
\begin{lem}\label{l:conv}
Let $U_{\mF}$, $U_{\mG}\subset \cP$ be two open sets, and let $\mF_n:U_{\mF}\to \cP$, $\mG_n:U_{\mG}\to \cP$ be two sequences of continuous maps, which converge  to functions $\mF$ and $\mG$ (resp.), uniformly on $U_\mF$ and~$U_\mG$ respectively. 
Let $\mK\subset U_\mG$ be a compact set such that $\mG(\mK)\subset U_{\mF}$. Then there exist $n_0$ and an open neighbourhood~$U$ of $\mK$ in $U_{\mG}$ such that, for all $n\ge n_0$, the composition $\mF_n\circ \mG_n$ is defined on $U$ and converges to $\mF\circ \mG$ uniformly on $U$, as $n\to\infty$.
\end{lem}

Fix the space $\cQ$ and the compact neighbourhood $\cQ_0$ satisfying the assumptions~\ref{a:everywhere} and~\ref{a:convergence}. Consider the compact segment
\[
\mK_0:=\{\bm_c \mid c\in [-\delta',0] \}.
\]
Up to taking a larger $\cQ_0$, we can suppose that it contains $\mK_0$ and the cut-off measures
$\Phi_{b_0}[\mK_0]$.
Due to the assumptions~\ref{a:everywhere} and~\ref{a:convergence}, given $\mu\in \cQ$, one has the convergence
\[\lim_{\ell\to\infty}\Phi^\ell[\mu]= \bm_{\tilde{c}(\mu)},\]
hence the operators $\{\Phi_{\bb(b_0,\ell)}\}_{\ell\in\N}$ pointwise converge on $\cQ$ to the operator $\mF:\mu\mapsto \bm_{\tilde{c}(\Phi_{b_0}[\mu])}$ as~$\ell\to\infty$, and the convergence is uniform on $\cQ_0$. 

Note that on the line $\{\bm_c\}_{c\in\R}$ the limit map $\mF$ acts by $\bm_c \mapsto \bm_{S_{b_0}(c)}$. Hence, considering our choice for $b_0$, the image $\mF [\mK_0]$ is contained in the segment
\[
\mK_1:=\{\bm_c \mid c\in [-\delta'',0] \}.
\]
Up to extending the neighbourhood $\cQ_0$, we can suppose that it contains $\mK_1$. Using again the assumptions~\ref{a:everywhere} and~\ref{a:convergence}, given $\mu\in \cQ$, one has the convergence
\[\lim_{\ell\to\infty}\Phi_{\bb(b_1,\ell)}[\mu]= \bm_{\tilde{c}(\Phi_{b_1}[\mu])},\]
hence the operators $\Phi_{\bb(b_1,\ell)}$ converge to the operator $\mG:\mu\mapsto \bm_{\tilde{c}(\Phi_{b_1}[\mu])}$ as $\ell\to\infty$, and the convergence is uniform on $\cQ_0$. 

Applying Lemma~\ref{l:conv} to every composition in
\[
\mG_{\ell}:=\underbrace{\Phi_{\bb(b_1,\ell)}\circ \dots \circ \Phi_{\bb(b_1,\ell)}}_{N\text{ times}} \circ \Phi_{\bb(b_0,\ell)},
\]
we see that there is a neighbourhood $U$ of $\mK_0$ on which $\mG_{\ell}$ converges uniformly to $\mG^N \circ \mF$. Finally, by construction 
\[
\mG^N  \mF [\ddelta_{+\infty}] = \bm_{S_{b_1}^N(c_0)},
\]
and this measure is $\eps/2$-close to $\bm$, because of the inequality $|S_{b_1}^N(c_0)|<\eps_0$ and the choice of $\eps_0$. Hence, 
for every  sufficiently large $\ell$, the measure $\mG_{\ell}[\ddelta_{+\infty}]$ is $\eps$-close to $\bm$.

On the other hand, for any $c'\in [-\delta',0]$ we have
\[
\mG^N  \mF [\bm_{c'}] = \bm_{c},
\]
where $c= S_{b_1}^N(S_{b_0}(c')) \in [-\delta,0]$. Due to the uniform convergence (and hence uniform continuity), there exists $\eps'>0$ such that for every sufficiently large $\ell$, for every $c'\in[-\delta, 0]$ and every $\mu'\in\cQ_0$ that is~$\eps'$-close to $\bm_{c'}$, we have that $\mG_\ell[\mu']$ is $\eps$-close to $\bm_c$, for some $c\in [-\delta,0]$. 

Taking $\ell_0=\ell_1=\ell$ sufficiently large so that both conclusions above hold, we see that the block~$\bB$, obtained by juxtaposition of $N$ copies of $\bb(b_1,\ell)$ and one $\bb(b_0,\ell)$, satisfies the conclusions of Proposition~\ref{p:blocks}.

With this, the proof of Theorem~\ref{t:main} is accomplished.

\section{Random metrics on hierarchical graphs}\label{s:metrics}
\paragraph{Background --}\label{s:background}

Let us briefly recall the setting of the problem (see~\cite[\S\,8.1]{KKT}). Assume that we are given a graph~$\Gamma$ with a chosen orientation on the edges and with two marked vertices,~$I$ and~$O$, and a law~$m$ supported on~$\R_+=(0,+\infty)$. One can construct a sequence of marked weighted graphs $(\Gamma_n, I,O)$ in the following way:
\begin{itemize}
\item $\Gamma_0$ is the edge between and $I$ and $O$, whose length is $1$.
\item For each $n\ge 1$, we replace each edge $e$ of $\Gamma_{n-1}$ by a rescaled copy of the basic graph $\Gamma$; here~$I$ is glued to the beginning of the edge, and $O$ to the end. The length of every edge of the new copy is taken to be $\xi_e \cdot |e|$, where $|e|$ is the length of the replaced edge $e$, and all random variables $\xi_e$ are i.i.d.~with the law~$m$. 
\end{itemize}
This process is illustrated on Fig.~\ref{f:diamond} for the diamond graph. Each graph $\Gamma_n$ is equipped with its graph-distance $d_n$, and the vertices of $\Gamma_n$ are naturally included in the set of vertices of~$\Gamma_{n+1}$. The question is whether there exists a normalization constant $\lambda$ such that the sequence of distances $\lambda^n d_n$ converges to a distance (for instance, on the set which is the union of the vertices of all the $\Gamma_n$'s). At a combinatorial level, the limit space is the \emph{hierarchical graph} associated with $(\Gamma,I,O)$.

\begin{figure}
\[\includegraphics[scale=1]{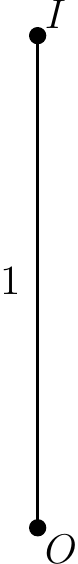} \qquad \includegraphics[scale=1]{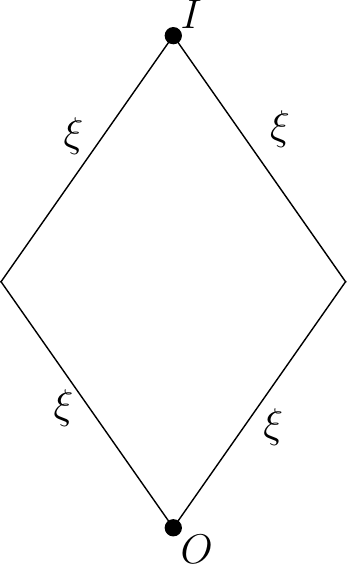}\qquad \includegraphics[scale=1]{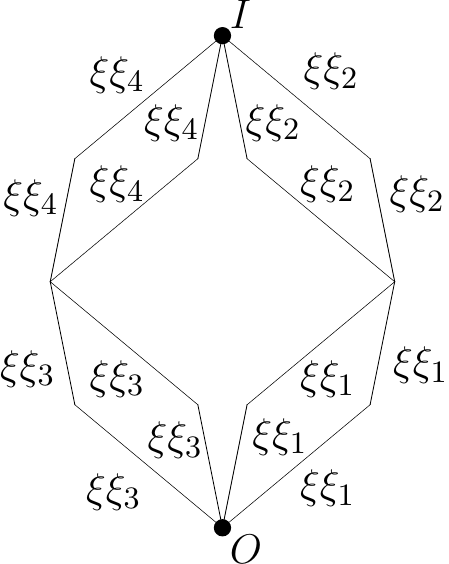}\qquad\includegraphics[scale=1]{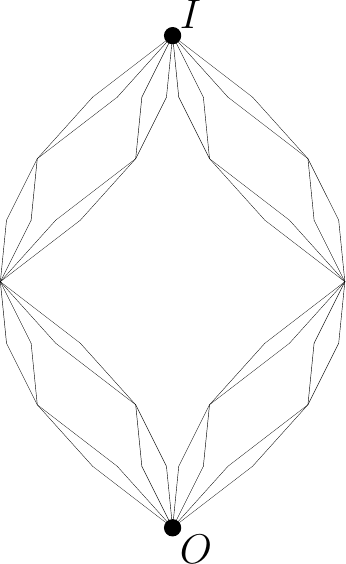}\]
\caption{Multiplicative cascade on a hierarchical diamond graph.}\label{f:diamond}
\end{figure}

\begin{rem}
An important remark is that for a generic law for the random variable~$\xi$, with no normalization by $\lambda$, the distances obtained by 
this process do not converge. Indeed, if we replace the law of $\xi$ by the law of $(1+\eps) \xi$, the distances $d_n$ obtained on the $n$th step 
will be multiplied by~$(1+\eps)^n$. Thus, if we have convergence for the initial law without a normalization by $\lambda$, the metrics 
obtained with any $\eps>0$ will explode, while with any $\eps<0$ will collapse.
\end{rem}

One can also consider the process in the other direction, saying that we are replacing each of the edges of $\Gamma$ by an independent copy of $\Gamma_{n-1}$, and then rescale the result by a random common factor~$\xi$. This process leads us to a RDE, which relates the $IO$-distance $\hat{X}$ (that is, length of the shortest path between these two points) to the analogous distances $\hat{X}_e$'s in the children graphs:
\begin{equation}\label{eq:Gamma}
\hat{X}\stackrel{d}{=}\lambda\hat{\xi} \cdot \min_{\pi} \sum_{e\in\pi} \hat{X}_e,
\end{equation}
where the minimum is taken over all the paths $\pi$ joining $I$ to $O$ in~$\Gamma$.
For instance, for the diamond graph one has
\[
\hat R_{\lambda}(\hat{X}_1,\dots,\hat{X}_4; \hat{\xi})= \lambda \hat{\xi} \cdot \min(\hat{X}_1+\hat{X}_2,\hat{X}_3+\hat{X}_4).
\]
Passing to the logarithmic coordinates $X_i=\log \hat{X}_i$, $\xi=\log \lambda \hat{\xi}$, the RDE~\eqref{eq:Gamma} turns into
\begin{equation}\label{eq:Gamma-log}
R(X_i;\xi) = \xi+ \log \min_{\pi} \sum_{e\in \pi} \exp(X_e).
\end{equation}
These are the RDEs of Theorem~\ref{t:applying}. Note that we have incorporated the (unknown) critical drift $\log \lambda$ in 
the law of $\xi$; its existence for non-pivotal graphs is guaranteed by~\cite[Thm.~1]{KKT}. In particular, in the statement of Theorem~\ref{t:applying} we are taking $\xi=\mN(a,\sigma^2)$, where $a$ is chosen for a given 
$\sigma^2$ in such a way that~\eqref{eq:Gamma-log} admits a solution that is almost surely finite.

However, to use~\cite[Thm.~1]{KKT} (and other results \emph{ibid.}), we need to add a restriction on the geometry of the graph:
we say that a graph is \emph{non-pivotal} if there is no straight edge between $I$ and $O$ (a \emph{shortcut} edge), nor edges whose removal disconnect the graph (\emph{bridge} edges). For instance, the diamond graph is non-pivotal and, with the figure-eight graph (which was the leading example in~\cite{KKT}), it is the simplest graph which is non-pivotal. An example of pivotal graph is exhibited by what we call the \emph{racket} graph, which has bridge edges (see Fig.~\ref{f:racket}).
\begin{figure}
\[\includegraphics[scale=1]{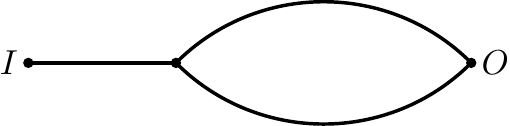}\]
\caption{The racket graph.}\label{f:racket}
\end{figure}

The different behaviour of the model on pivotal and non-pivotal graphs is partly detected by the \emph{percolation function} $\theta_\Gamma(p)$ that is defined as the probability that $I$ and $O$ are connected in $\Gamma$ when considering a Bernoulli percolation of parameter $p$ (see \cite[Def.~12]{KKT}). For non-pivotal graphs, $p=0$ and $p=1$ are super-attracting fixed points, while for pivotal graphs there is one which is (topologically) repelling.

\begin{rem}
This problem of finding a limit random metric was motivated by the so-called (mathematical) 2D Liouville field theory (see \cite[\S~3.3]{KKT}). Hierarchical graph have the advantage that the problem can be formulated in the language of RDEs. An interesting recent work, closer to the original problem, has been done recently by Ding and Dunlap \cite{DD}: in some special cases (high temperatures) they show the existence of a non-trivial random metric on the square, making approximations with random metrics on discrete square grids.
\end{rem}

\paragraph{Establishing assumptions --}

For any non-pivotal hierarchical graph, \ref{a:incr} is evident, as well as \ref{a:translation}: changing $X_i$ to $X_i+c$ multiplies the 
distances $\hat{X}_i$ by $e^c$, that multiplies the resulting distance $\hat{X}$ by the same constant (being a simple rescaling), and finally changes the value of $R$ by $\log e^c=c$.

For the remaining conditions, from~\cite[Lemma~2 and Prop.~12]{KKT} we can choose any sufficiently small $\alpha>0$ such that $\alpha$ and $1-\alpha$ belong respectively to the basins of attraction of $0$ and $1$ for the percolation function~$\theta_\Gamma$. Possibly choosing a smaller value, we suppose that for such an~$\alpha$, Proposition~\ref{l:shift} below holds. We set
\[\cQ=\{\mu\in \cP \mid \mu(\{\pm\infty\})\le \alpha\}.\]
and we equip it with a metric $d_\cQ$ defining the weak-$*$ topology. It is straightforward to verify that~$(\cQ,d_\cQ)$ satisfy the assumption~\ref{a:space}.

Condition~\ref{a:stationary} is given by~\cite[Thm.~1]{KKT} (as well as by the choice of the normalization constant~$\lambda$).
Pointwise convergence in condition~\ref{a:everywhere} is implied by~\cite[Thm.~2]{KKT}: if the initial measures belong to~$\cP_0\cap \cQ$, then it is exactly the statement of that theorem, otherwise, it is enough to choose~$\alpha$ from~\cite[Prop.~5]{KKT} (which must satisfy~\ref{a:everywhere}). For assumption~\ref{a:infinity}, uniform convergence on a neighbourhood of~$\bm$ can be proved with arguments similar to those for~\cite[Thm.~2]{KKT}, as we will explain with Lemma~\ref{l:uniform} below. Finally, condition \ref{a:superexp}  will be verified with Lemma~\ref{l:decrease_superexp}.

\paragraph{Reminder: upper and lower bounds --}\label{s:upper-lower}
Before establishing uniform convergence and superexponential decrease, 
let us remind some of the conclusions and arguments from~\cite[\S\,6]{KKT}, that we will be using and extending here. 
First, we have the following convergence result, that is given by a part of the conclusions of~\cite[Thm.~2]{KKT}, rewritten in the logarithmic scale:

\begin{prop}\label{p:t2}
Assume that the graph $\Gamma$ is non-pivotal and that the law $m$ of $\xi$ is given by $m=\rho(x) dx$, where $\rho:\R\to (0,+\infty)$ is continuous. 
Let $\bm\in\cP_0$ be a solution to the corresponding RDE. Then, for any initial measure $\mu\in \cP_0$ there exists $c=:\tc(\mu)$ such that $\Phi^n[\mu]\to \bm_c$ as $n\to\infty$.
\end{prop}
This proposition gives us the pointwise convergence needed in~\ref{a:everywhere}, although for measures from~$\cP_0$ only. However, the same arguments that were used in~\cite{KKT} imply the convergence in $\cQ$,
as well as its uniformity on some neighbourhood~$\cQ_0$ of~$\bm$. We recall here the main arguments. From now on, to fix notation, we shall write~$F_\mu$ for the distribution function of a measure~$\mu$.

\begin{dfn}\label{def:class}
A measure $\mu$ is of \emph{(upper) class $(\alpha,\delta)$} if the partition functions $F_{\mu}$ and $F_{\bm}$ satisfy the inequalities
\begin{itemize}
\item $F_{\mu}(x)\le \Fbm(x)$ on $[\kal,\kone]$,
\item $F_{\mu}(x)\le \Fbm(x)+\delta$ on $[-\infty,\kal)\cup (\kone,+\infty]$,
\end{itemize}
where $\kappa_{\alpha}$ and $\kappa_{1-\alpha}$ are the $\alpha$- and $(1-\alpha)$-quantiles of the measure $\bm$ respectively (see Fig.~\ref{f:c-a-d}).
We denote the set of such functions by $\mC_{\alpha,\delta}$.
\end{dfn}
\begin{figure}
\[
\includegraphics[width=.7\textwidth]{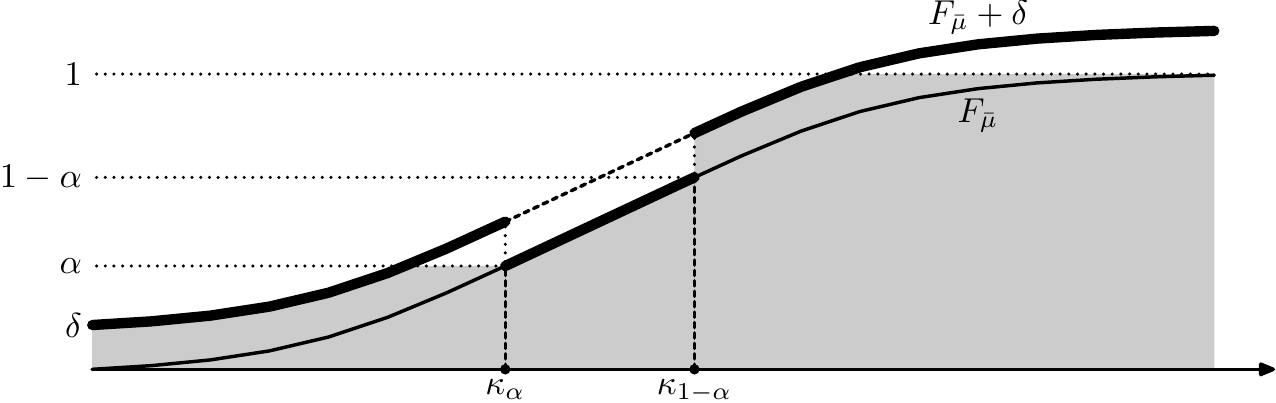}
\]
\caption{For a measure $\mu$ of class $(\alpha,\delta)$, the graph of its distribution function stays below the bold curves,
and by monotonicity it is contained in the grey-filled area.}\label{f:c-a-d}
\end{figure}
Note that there is no reason to consider 
$\delta>\alpha$: the class $\mC_{\alpha,\delta}$ for $\delta>\alpha$ coincides with~$\mC_{\alpha,\alpha}$. 
We also consider the family of translation operators $T_r$: 
\begin{dfn}
For any $r\in\R$, the operator $T_r$ is defined by sending any measure $\mu\in\cP$ to the  shifted measure $T_r[\mu]:=\mu(\cdot-r)$.
\end{dfn}

Two propositions from~\cite{KKT}  describe the iterations of measures, when starting from measures of some class $(\alpha,\delta)$. 
The first of them states that the $\Phi$-image of a measure of class $(\alpha,\delta)$, after a slight translation, belongs to an even ``better''
class:
\begin{prop}[\mbox{\cite[Prop. 5]{KKT}}]\label{l:shift}
For every  sufficiently small $\alpha$, there exists a constant $L=L(\alpha)>0$ such that for every $\delta\in (0,\alpha]$ the operator 
$T_{L\delta}\Phi$ sends $\mC_{\alpha,\delta}$ in $\mC_{\alpha,\delta/2}$.
\end{prop}
In short, a coupling argument allows to get a uniform estimate for the partition functions $F_{\Phi[\mu]}(x)\le F_{\bm}(x)+\const \cdot \delta$; a translation by $L\delta$ then removes the correction $\const \cdot \delta$ in the center (as the density of $\bm$ is positive, in a compact domain it is bounded away from zero by some constant~$\frac{1}{L}$). On the other hand, the tail of the image depends on the tail of the starting measure in a way resembling a contracting operator. Roughly speaking, the reason is that the graph is non-pivotal, so it takes at least two large (``parallel'') edges to make a large distance in the glued graph, and it takes at least two short (``consecutive'') edges to make a short distance in the glued graph: this leads to a nearly squaring (at least) of the small probabilities of these events. We refer the reader to~\cite[\S\,6.1]{KKT} for details.

From now on, let us fix any $\alpha$ and $L$ such as in Proposition~\ref{l:shift}. Applying this proposition inductively provides us with the following proposition:
\begin{prop}[\mbox{\cite[Prop.~7]{KKT}}]\label{c:upper}
Given a measure $\mu\in\cP$  of class $(\alpha,\delta)$, for any $n$ one has
that for any $x\in\R$,
\[
F_{\Phi^n[\mu]}(x)\le F_{\bm_{-2L\delta}}(x)+\delta/2^n.
\]
\end{prop}
We recall the following definition, motivated by Proposition~\ref{c:upper} (cf. \cite[\S\,6.2, Def.~8]{KKT})
\begin{dfn}
A measure $\mu$ is \emph{asymptotically upper bounded} (resp. \emph{lower bounded}) by a translation~$\bm_c$ of the stationary measure~$\bm$, if for any $\eps>0$ there exists $n_0$ such that for any $n>n_0$ and~$x\in\R$,
\[
F_{\Phi^n[\mu]}(x)\le F_{\bm_c}(x)+\eps \quad \text{(resp. } \, F_{\Phi^n[\mu]}(x)\ge F_{\bm_c}(x)-\eps \text{).}
\]
\end{dfn}

\paragraph{Uniform convergence and superexponential decrease estimate --}\label{s:uniform}
We can now show that the assumption \ref{a:convergence} holds:

\begin{lem}\label{l:uniform}
For any compact set $\mK$ in the space $\cQ$ the convergence 
\[
\Phi^n[\mu] \to \bm_{\tc(\mu)} 
\]
is uniform on $\mK$ and the function $\tc$ is continuous on~$\mK$.
\end{lem}

\begin{proof}
Take any measure $\mu\in\mK$. By Proposition~\ref{p:t2} (or \cite[Thm.~2]{KKT}) and the definition of $\tc$, 
the measures $\Phi^n[\mu]$ converge to the measure $\bm_{\tc(\mu)}$ and hence the 
distribution functions of $\Phi^n[\mu]$ converge uniformly to the (continuous) distribution function of $\bm_{\tc(\mu)}$. In 
particular, for an arbitrarily small~$\delta>0$ there exists $n_0$ such that the partition function of $\Phi^{n_0}[\mu]$
is $\delta/2$-close to the one of $\bm_{\tc(\mu)}$. 

For what follows, let us assume~$\tc(\mu)=0$ (otherwise we translate everything by~$-\tc(\mu)$ and use the translation-equivariance).
Use this choice and apply again (as in Proposition~\ref{l:shift}) a translation by~$L\delta$ to handle the ``core'' part. The distribution function of the obtained measure $T_{L\delta} \Phi^{n_0}[\mu]$ then belongs to the class $\mC_{(\alpha,\delta)}$. Moreover, 
due to the continuity of $\Phi^{n_0}$, the same holds for any measure~$\mu'$ belonging to some neighbourhood~$U$ 
of the initial measure~$\mu$. 

The application of Proposition~\ref{c:upper} gives that for any $n>n_0$ and for any $\mu'\in U$ the image 
\[T_{L\delta}\Phi^n[\mu']=\Phi^{n-n_0}[T_{L\delta}\Phi^{n_0}[\mu']]\]
is asymptotically upper bounded by $\bm_{-2L\delta}$, and thus for any $n>n_0$ and $x\in\R$,
\[
F_{\Phi^n[\mu']}(x)\le T_{-3L\delta}[F_{\bm}](x)+\delta/2^{n-n_0}.
\]

In fact, as $\Phi^{n_0}[\mu]$ is $\delta/2$-close to $\bm$, working with measures in \emph{lower classes $(\alpha,\delta)$}, we can repeat the same arguments to find asymptotically lower bounds (possibly decreasing the choice of $\alpha$): for any $n>n_0$ and $\mu'\in U$, for any $x\in\R$,
\[
F_{\Phi^n[\mu']}(x)\ge T_{3L\delta}[F_{\bm}](x)-\delta/2^{n-n_0}.
\]
Then passing to the limit in these upper and lower bounds, since we know that $\Phi^n[\mu']$ converges to~$\bm_{\tc(\mu')}$, we have
\[
T_{3L\delta}[F_{\bm}](x)\le F_{\bm_{\tc(\mu')}}(x) \le T_{-3L\delta}[F_{\bm}](x)\quad\text{for any }x\in\R
\]
and thus $|\tc(\mu')|\le 3L\delta$,
which proves the continuity part of the lemma (cf.~\cite[Prop.~8]{KKT}). Finally, we see that for every measure $\mu'\in U$, the images $\Phi^{n}[\mu']$, for $n\ge n_0$, stay uniformly close to the corresponding translates of~$\bm$. Together with the compactness of $\mK$, this ensures the desired uniformity.
\end{proof}

The next lemma guarantees the last condition~\ref{a:superexp}:

\begin{lem}\label{l:decrease_superexp}
The function $\Delta$ decreases superexponentially: for every $\beta>0$,
\[
\Delta(a)=o\left (e^{-\beta a}\right )\quad \text{as }a\to+\infty.
\]
\end{lem}
\begin{proof}
According to \cite[Lemma~19]{KKT}, the tail distribution function $f_{\bm}(x)=1-F_{\bm}(x)$ of $\bm$ has a superexponential decrease: for any $\beta>0$,
\[
f_{\bm}(s)=o\left (e^{-\beta s}\right )\quad\text{as }s\to+\infty.
\]
Consider a cut-off measure $\Phi_a[\bm]$ for any $a>\kone$. This measure is of class $(\alpha,f_{\bm}(a))$.
Applying Proposition~\ref{c:upper} in the same way as in the proof of Lemma~\ref{l:uniform}, we see that $|\tc(\Phi_a[\bm])|\le 2Lf_{\bm}(a),$
while by definition $\Delta(a)=-\tc(\Phi_a[\bm])$. Thus, we get the desired upper bound
\[
\Delta(a)\le 2L f_{\bm}(a) = o(e^{-\beta a}) \quad\text{as } a\to+\infty,
\]
for any $\beta>0$.
\end{proof}

\section{Mean-field optimization problems}\label{s:general}

We shall only describe briefly the combinatorial models: the references mentioned within this section provide neat expositions.

\paragraph{Mean-field minimal matching in pseudo-dimension~$q$ --} 
Let $K_{n,n}$ be the complete $n\times n$ bipartite graph, with i.i.d.~edge-lengths. As the name suggests, the mean-field minimal matching problem consists in describing the law of the minimal total length (or cost) for a perfect matching. M\'ezard and Parisi, with methods borrowed from statistical physics \cite{MP0,MP2}, guessed the asymptotic behaviour of this random total cost.
After Aldous \cite{A-assign}, the problem can be formulated involving RDEs.

We resume notation from~\cite[\S\,7.4]{RDE} (see also \cite{zeta2}): let a pseudo-dimension $q>0$ be chosen. Take $\xi_1<\xi_2<\dots$ to be jump points of a Poisson
process with intensity $x^{q-1}dx$; namely, the expectation of the number of $i$'s such that $\xi_i<x$, is equal to $x^q/q$ for any $x>0$.

The mean-field minimal matching problem is associated with the RDE~\cite[Eq.~(94)]{RDE}:
\begin{equation}\label{eq:minimal}
X\stackrel{d}{=} \min_{1\le i<\infty} (\xi_i-X_i).
\end{equation}

\paragraph{Mean-field TSP and minimum weight $k$-factor --}
Let $k\ge 1$ be an integer. The previous mean-field optimization problem can be generalized (as done in \cite{K}) looking at $k$-factors instead of perfect matchings: a $k$-factor is a spanning $k$-regular subgraph. A $1$-factor is exactly a perfect matching. As above, the problem of describing (asymptotically) the law of the minimal $k$-factor can be formulated using RDEs: with the same notations as above the equation is
\begin{equation}\label{eq:TSPk}
X\stackrel{d}{=} {\min_{1\le i<\infty}}^{[k]} (\xi_i-X_i),
\end{equation}
where $\min^{[k]}$ stays for the $k$th smallest element of the set.

When $k=2$, this RDE is also associated with the solution to the mean-field approximation to the travelling salesman problem (see \cite[Eq.~(95)]{RDE} and \cite{frieze}): this is rather intuitive, for a $2$-regular spanning subgraph is union of closed non-self-intersecting loops.

\paragraph{Tail distribution functions: cavity equations --}
From statistical physics, the equations~\eqref{eq:minimal} and \eqref{eq:TSPk} take the name of \emph{cavity equations}, after M\'ezard and Parisi (see for instance \cite{MP1,MP2}). They are sometimes expressed in terms of \emph{tail distribution functions} and it is worth to present them in this form: it will prepare the reader for the proof of Theorem~\ref{t:TSP}.

First, let $\{\xi_i\}$ be a Poisson point process with intensity $\rho( x )dx$, and $\{X_i\}$ i.i.d.~random variables of law $\mu$, independent of $\{\xi_i\}$. Then the point process $\{\xi_i-X_i\}$ is also given by a Poisson process, this time, with intensity $\tilde\rho(x)dx$, given by a (subtractive) convolution
\[
\tilde\rho(x)=\int^{+\infty}_{-\infty}\rho(y+x)d\mu(y)
\]
(this is the so-called ``displacement theorem'' \cite[\S~5.5]{kingman}, see also \cite[Lemma~5]{zeta2}).

In particular, the number of points of this new process smaller than any given $x$ is a Poisson random variable with parameter $\tilde\lambda(x)=\int_{-\infty}^x\tilde\rho(y)dy$.
Thus, if $f(x)=\bP(X_i>x)$ is the tail distribution function of some $X_i$, then the parameter of the new Poisson random variable is given by
\[
\tilde\lambda(x)=I[f](x):=\int_{-\infty}^{+\infty} \rho(x+y)f(y)dy .
\]
Next, a Poisson random variable with parameter $\lambda$ is less than $k$ with probability 
\begin{equation}\label{eq:P}
\mP_k(\lambda):=e^{-\lambda} \sum_{j=0}^{k-1} \frac{\lambda^j}{j!}, 
\end{equation}
and thus the tail distribution function of ${\min_{1\le i<\infty}}^{[k]} (\xi_i-X_i)$ is equal to $\mP_k(I[f](x))$.

A tail distribution function $f$ solves the cavity equation (that is, $f$ is the tail distribution function of a random variable solution to the RDE \eqref{eq:TSPk}), if and only if
\[
f=P_k\circ I[f].
\]
When the pseudo-dimension $q$ is equal to $1$ in \eqref{eq:minimal}, Aldous proved in \cite{zeta2} that the tail distribution function of the logistic distribution
\[
f(x)=\frac{1}{1+e^{x}}
\]
solves the corresponding cavity equation. For other values of parameters no such explicit solution is known (and, most probably, it never admits any 
reasonable analytic expression).

\paragraph{Remarks --}
These distributional equations differ from those that we have studied in the previous section, and in several aspects. First, $\xi$ is no longer a real-valued random variable, as it takes values in the space of sequences~$\{\xi_i\}$. Second, there is an infinite number of $X_i$'s that are used. Third, the relation function $R$ is monotone non-increasing instead of monotone non-decreasing. 

The first issue is not so important: we have never used that $\xi$ was taking real values  in the proof of Theorem~\ref{t:main}; the arguments from the proof of Theorem~\ref{t:applying} will require some slight modifications, but there is nothing substantial to be changed.
Also, as we explained in Section~\ref{s:cut-off}, dealing with an infinite number of variables requires some extra care, but it can be managed. 

Lastly, the third issue can be easily solved by passing to the square of the map~$\Phi$, doing two iterations at once. For example, in \eqref{eq:minimal}, this leads to the RTP
\begin{equation}\label{eq:two}
X\stackrel{d}{=} \min_{1\le i<\infty} (\xi_i- \min_{1\le j<\infty} (\xi_{ij}-X_{ij})) =\min_{1\le i<\infty} (\xi_i+ \max_{1\le j<\infty} (X_{ij}-\xi_{ij})).
\end{equation}
This RTP is translation-equivariant, and it looks natural to try to apply the methods of~\cite{KKT}.
However, one should note that its features look more like those of the RDE corresponding to a \emph{pivotal} graph. Indeed, if a measure $\mu$ has an atom at $+\infty$ with an arbitrarily small weight, its image by the operator associated with~\eqref{eq:minimal} is automatically $\ddelta_{-\infty}$, and its image by the one associated with~\eqref{eq:two} is~$\ddelta_{+\infty}$.
Moreover, this dramatic collapse occurs even for measures $\mu$ whose tails at~$+\infty$ do not have sufficiently good decay.

To overpass this problematic issue, the solution is to restrict to a space $\cQ$ of measures whose tails have a sufficiently good integrability condition, as done first in~\cite{BP} and then in~\cite{K,salez}. Indeed, note that from the expression~\eqref{eq:P} we have the asymptotic equivalences
\begin{equation}\label{eq:poisson1}
\mP_k(\lambda)=O(\lambda^{k-1} e^{-\lambda})\quad \text{as } \lambda\to +\infty
\end{equation}
and
\begin{equation}\label{eq:poisson2}
1-\mP_k(\lambda)=e^{-\lambda}\sum_{j=k}^{+\infty} \frac{\lambda^j}{j!} = O(\lambda^{k} )\quad  \text{as } \lambda\to 0.
\end{equation}
Thus one can obtain good tail decays on the distribution function $\mP_k(I[f](x))$, which corresponds to the image under the operator associated with the RDE. 
The tail decays also allow to define an adapted metric $d_{\cQ}$ on $\cQ$, leading to an appropriate variation of the classes $\mC_{\alpha,\delta}$, so that the methods of \cite{KKT} will work.
We shall come back to this in \S\,\ref{par:last}.

\paragraph{Solutions to the equations --}
Actually, a cut-off method has already been used in \cite{W-match,salez,larsson} and \cite{K} to solve the RDEs \eqref{eq:minimal} and \eqref{eq:TSPk} respectively. Combining these works one can state:

\begin{thm}[Shah--Salez, W\"astlund, Khandwawala, Salez]\label{t:k}
Let $q\ge 1$ and let $\{\xi_i\}$ be a Poisson point process with intensity $x^{q-1}$. Let $k\ge 1$ be an integer. 
The equation~\eqref{eq:TSPk}
\[
X\stackrel{d}{=} {\min_{1\le i<\infty}}^{[k]} (\xi_i-X_i),
\]
admits a unique solution $\bm$. The Dirac measures $\ddelta_{\pm\infty}$ and the translates~$\bm_c$'s, $c\neq 0$, of the stationary measure $\bm$, are periodic points of order $2$ for $\Phi$, and any measure $\mu\in\cP$ converges to one of these periodic orbits.
Moreover, the measures $\bm^{a}$ that are solutions to 
the cut-off equations
\[
X\stackrel{d}{=} \min\left ({\min_{1\le i<\infty}}^{[k]} (\xi_i-X_i),a\right ),
\]
converge to $\bm$ as $a\to+\infty$.
\end{thm}

\begin{rem}
For pseudo-dimension $0<q<1$, as a first partial result, Larsson \cite{larsson} proved the existence of unique, globally-attractive solutions to the cut-off equations. It would be interesting to see if our methods apply to these cases.
\end{rem}

In the remaining part, we inspect the proofs in \cite{BP,K,salez}  and verify the assumptions \ref{a:incr}--\ref{a:superexp}.

\paragraph{Verifying assumptions --}\label{par:last}

Let $q> 1$ and $k\in\N$ be fixed, let $\Phi$ be the operator associated to the RDE \eqref{eq:TSPk}. Given any measure $\mu\in\cP$, we denote by $f_\mu$ its tail distribution function.

As already noticed, if we want to stay within our setting of monotone RDEs, we have to iterate the recursion twice as in \eqref{eq:two}, and actually it is better to consider both even and odd iterations: the former are non-decreasing, the latter non-increasing.
The assumptions \ref{a:incr}--\ref{a:superexp} will be verified for~$\Phi^2$, namely the squared RDE \eqref{eq:two}.

\smallskip
 
We have already mentioned that for a given measure $\mu$ its tail distribution function $f_\mu$ does not verify the integrability condition 
\begin{equation}\label{eq:tail_cond}
\int_0^{+\infty}y^{q-1}f_\mu(y)dy<+\infty,
\end{equation}
then applying $\Phi$ collapses the measure: $\Phi[\mu]=\ddelta_{-\infty}$ and $\Phi^2[\mu]=\ddelta_{+\infty}$.

We define $\widehat\cQ$ to be the space of measures
\[
\widehat\cQ=\{\mu\in \cP_0\mid f_\mu\text{ verifies }\eqref{eq:tail_cond}\}
\]
and we equip it with the distance
\begin{equation}\label{eq:distance}
d_{\cQ}(\mu,\nu)=\sup_{x\in \R}|f_\mu(x)-f_\nu(x)|e^{C_{q}|x|},
\end{equation}
where $C_{q}$ is some (sufficiently large) constant, depending on  $q$, that will be fixed by Proposition~\ref{l:last}.

With respect to the distance $d_\cQ$, there are measures in $\widehat\cQ$ that are at infinite distance, whence we define
\[
\cQ:=\{
\mu\in \widehat\cQ\mid d_{\cQ}(\mu,\bm)<+\infty
\}.
\]

The study of convergence to the solution in \cite[Thm.~5.1]{BP} is done for the case $k=q=1$ for which the cut-off is not needed, but it can be adapted to the general case:

\begin{thm}[Shah--Salez]\label{t:BP}
Let $q\ge 1$ and let $\{\xi_i\}$ be a Poisson point process with intensity $x^{q-1}$. Let $k\ge 1$ be an integer. Let $\bm$ be the solution to the RDE \eqref{eq:TSPk} and 
$\mu\in\cQ$ any measure. Then there exists a constant $c=\tc(\mu)\in\R$ such that
\[
\Phi^{2n}[\mu]\to\bm_c\quad\text{and}\quad\Phi^{2n+1}[\mu]\to\bm_{-c}\quad\text{as }n\to\infty.
\]
Moreover, with respect to the distance $d_\cQ$, the function $\tc$ is continuous and the convergences are uniform.
\end{thm}

Therefore, from now on, we shall only consider measures satisfying \eqref{eq:tail_cond}
and actually after four iterations we can restrict even more the space on which we study the behaviour of $\Phi$. Indeed, using the asymptotic equivalences  \eqref{eq:poisson1} and \eqref{eq:poisson2}, proceeding as in \cite[Lemma~3]{K} and \cite[Prop.~3.1]{salez}, we have:
\begin{lem}\label{l:td}
Let $n\ge 4$. For any measure $\mu\in\cQ$, the image $\Phi^n[\mu]$ has the tail decays
\begin{align}
f_{\Phi^n[\mu]}(x)&=O\left(x^{q(k-1)}e^{- x^q} \right)\quad\text{as }x\to+\infty.  \label{eq:tail_cond1} \\
1-  f_{\Phi^n[\mu]}(x)&=O\left(|x|^{qk(k-1)}e^{ -k |x|^q} \right)\quad\text{as }x\to-\infty. \label{eq:tail_cond2}
\end{align}
In particular, when $q>1$, the stationary measure has a superexponential tail decay.
\end{lem}

\smallskip

The conditions \ref{a:incr} and \ref{a:translation} are immediately verified for the square $\Phi^2$.
The space $(\cQ,d_\cQ)$ for condition~\ref{a:space} has been chosen and continuity of operators can be checked using the explicit expression at the level of tail distribution functions, namely
\[
f_{\Phi[\mu]}=P_k\circ I[f_\mu]
\]
(this avoids dealing with the continuity of the function $R$ defined on the infinite product $[-\infty,\infty]^\N$). 
The assumption~\ref{a:stationary} follows from Theorem~\ref{t:k} and Lemma~\ref{l:td}: the solutions $\bm_c$ are supported on~$\R$ and clearly have integrable tails. 
The assumption \ref{a:everywhere} is verified after Theorem~\ref{t:BP}.

\begin{rem}
In fact, in \cite{BP} Shah and Salez prove the convergence for the ``horizontal'' displacement which, together with the superexponetial tail decay, guarantees the convergence with respect to our metric $d_\cQ$.
\end{rem}

We need now to establish the assumptions~\ref{a:infinity} and~\ref{a:superexp}. 
Proceeding in the same way as in \S\S\,\ref{s:upper-lower}--\ref{s:uniform}, we would like to use a series of upper and lower bounds for the distribution functions 
that behave nicely under the iterations (compare with Definition~\ref{def:class}, Propositions~\ref{l:shift} and~\ref{c:upper}), and to combine them with the decay bounds on the stationary measure $\bm$ to establish the superexponential decay of the function~$\Delta$ (cf.~Lemma~\ref{l:decrease_superexp}).

However, the definition for the lower class $(\alpha,\delta)$ requires some adaptation for this case: trying to copy it directly, we stumble upon measures having atoms at $+\infty$, and as we have already discussed, such measures are sent to $\ddelta_{\pm\infty}$ in one or two iterations of the map $\Phi$. Thus, instead of changing the partition function $F_{\bm}$ by a constant~$\delta$, the allowed difference between the partition functions will depend on the point. 

That is,
given $M>0$ (that plays the role of $\alpha$ and that will be fixed by Proposition~\ref{l:last}) we introduce the perturbation function
\[
f_{pert}(x)=\begin{cases}
e^{-C_{q}|x|}, & |x|\ge M \\
0, & |x|<M,
\end{cases}
\]
where $C_{q}$ is the constant defining the distance $d_{\cQ}$ in \eqref{eq:distance} (that will be also fixed by Proposition~\ref{l:last}). We then define:

\begin{dfn}\label{d:adapted}
A measure $\mu\in\cQ$ is of \emph{adapted upper class $(M,\delta)$} if the tail distribution function~$f_{\mu}$ satisfies the inequality
\begin{equation}\label{eq:upper_class}
f_{\mu}(x)\le f_{\bm}(x)+\delta f_{pert}(x)\quad\text{for every }x\in\R,
\end{equation}
where~$f_{\bm}$ is the tail distribution function of the stationary measure~$\bm$ (see Fig.~\ref{fig:adapted}).

Analogously, a measure $\mu\in\cQ$ is of \emph{adapted lower class} $(M,\delta)$ if
\begin{equation}\label{eq:lower_class}
f_{\mu}(x)\ge f_{\bm}(x)-\delta f_{pert}(x)\quad\text{for every }x\in\R.
\end{equation}
\end{dfn}
\begin{figure}
\[
\includegraphics[width=.8\textwidth]{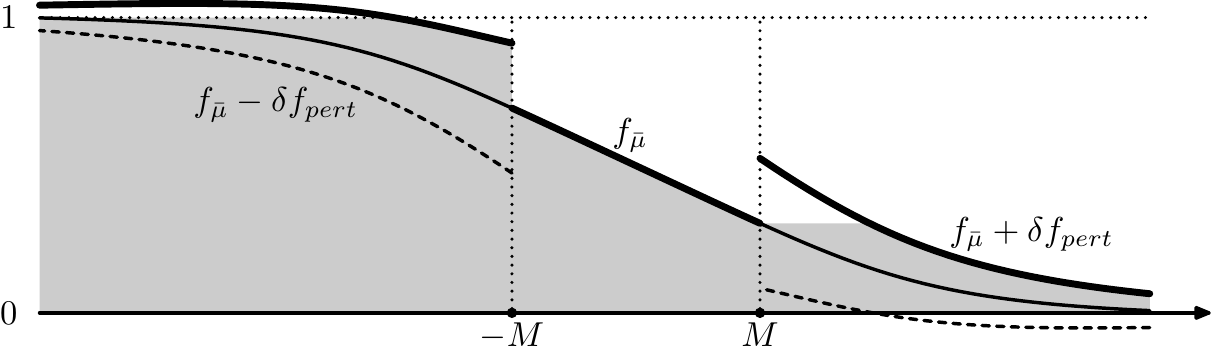}
\]
\caption{The adapted classes $(M,\delta)$. The grey zone is where any tail distribution function $f$ 
belonging to the adapted upper class $(M,\delta)$ is forced to stay by its monotonicity.}\label{fig:adapted}
\end{figure}
With this definition we can follow the strategy of \cite{KKT}.
We have a result analogue to Proposition~\ref{l:shift} (from which we
choose the constant $C_{q}$ and $M$):
\begin{prop}\label{l:last}
There exist positive constants $C_{q}$, $M$, $L>0$ and $\delta_0\in (0,1)$ such that the following holds: \begin{enumerate}
\item for any $\delta\in(0,\delta_0)$, if the measure $\mu\in\cQ$ is of adapted upper class $(M,\delta)$, then $T_{L\delta}\Phi[\mu]$ is of adapted lower class $(M,\delta/2)$;
\item for any $\delta\in(0,\delta_0)$, if the measure $\mu\in\cQ$ is of adapted lower class $(M,\delta)$, then $T_{-L\delta}\Phi[\mu]$ is of adapted upper class $(M,\delta/2)$.
\end{enumerate}
\end{prop}

Once this proposition is proved, we proceed as in Lemma~\ref{l:uniform} to establish~\ref{a:convergence} (where $\cQ_0$ can be 
taken, for instance,  to be the closure of the $\delta_0/2$-neighbourhood of $\bm$), and as in Lemma~\ref{l:decrease_superexp} in order to ensure \ref{a:superexp}.
The rest of this section is devoted to the proof of this proposition.

As in \cite{KKT}, we first find the good bounds for $f_{\Phi[\mu]}$ outside a compact set (Lemma~\ref{l:last1}), with linear (in~$\delta$) bounds on this compact set. Then we perform linear (in $\delta$) translations to guarantee that the functions $f_{T_{\pm L\delta}\phi[\mu]}$ verify the needed inequalities everywhere (Lemma~\ref{l:last2}). Notice that the inequalities in Lemma~\ref{l:last1} are in the opposite directions: an upper bound for $f_{\mu}$ becomes a lower bound for its $\Phi$-image, and \emph{vice versa}.

\begin{lem}\label{l:last1}
There exist positive constants $C_{q}$, $M$ and $K>0$ such that for every $\delta\in(0,1)$ the following holds:
\begin{enumerate}
\item \label{l:last1a} if the measure $\mu\in\cQ$ is of adapted upper class $(M,\delta)$, then $f_{\Phi[\mu]}$ satisfies
\[
f_{\Phi[\mu]}(x)\ge\begin{cases}
f_{\bm}(x)-K\,\delta&\text{for every }|x|< M,\\
f_{\bm}(x)-\frac \delta 2 f_{pert}(x)&\text{for every }|x|\ge M;
\end{cases}
\]
\item \label{l:last1b} if the measure $\mu\in\cQ$ is of adapted lower class $(M,\delta)$, then $f_{\Phi[\mu]}$ satisfies
\[
f_{\Phi[\mu]}(x)\le \begin{cases}
f_{\bm}(x)+K\,\delta&\text{for every }|x|< M,\\
f_{\bm}(x)+\frac \delta 2 f_{pert}(x)&\text{for every }|x|\ge M.
\end{cases}
\]
\end{enumerate}
\end{lem}

\begin{lem}\label{l:last2}
Let $C_q$, $M$ and $K>0$ be given by Lemma~\ref{l:last1}. There exist $L>0$ and $\delta_0\in (0,1)$ such that for every $\delta\in(0,\delta_0)$ the following holds:
\begin{enumerate}
\item if $\mu\in\cQ$ satisfies
\[
f_{\Phi[\mu]}(x)\ge\begin{cases}
f_{\bm}(x)-K\,\delta&\text{for every }|x|< M,\\
f_{\bm}(x)-\frac \delta 2 f_{pert}(x)&\text{for every }|x|\ge M;
\end{cases}
\]
then $T_{L\delta}\Phi[\mu]$ is of adapted lower class $(M,\delta/2)$;
\item if $\mu\in\cQ$ satisfies
\[
f_{\Phi[\mu]}(x)\le \begin{cases}
f_{\bm}(x)+K\,\delta&\text{for every }|x|< M,\\
f_{\bm}(x)+\frac \delta 2 f_{pert}(x)&\text{for every }|x|\ge M.
\end{cases}
\]
then $T_{-L\delta}\Phi[\mu]$ is of adapted upper class $(M,\delta/2)$.
\end{enumerate}
\end{lem}

The proof of Lemma~\ref{l:last2} is almost immediate, following the lines of \cite[Lemma~17]{KKT}. 
Namely, in the compact region $|x|<M$ the problem is handled by a translation, as the measure $\bm$ is absolutely 
continuous with positive continuous density. At the same time we have 
\[
f_{T_{L\delta}\Phi[\mu]}(x)\ge f_{\Phi[\mu]}(x)\ge f_{-T_{L\delta}\Phi[\mu]}(x)
\] 
for any $x$, and hence the inequality for $|x|\ge M$ holds automatically.

Lemmas~\ref{l:last1} and~\ref{l:last2} together imply Proposition~\ref{l:last}, and Lemma~\ref{l:last2} is already proven. 
It only remains to prove Lemma~\ref{l:last1}. We start with two auxiliary estimates:

\begin{lem}\label{l:pert}
For every $M>0$, the function
$
I[f_{pert}](x)=\int_{-x}^{+\infty}(x+y)^{q-1}f_{pert}(y)dy
$
satisfies:
\begin{equation}\label{eq:mu1}
I[f_{pert}](x)=\frac{\Gamma(q)}{C_{q}^q}f_{pert}(x)\quad\text{for every }x\le -M;
\end{equation}
and
\begin{equation}\label{eq:mu2}
I[f_{pert}](x)= O(x^{q-1})\quad\text{as }x\to+\infty, 
\end{equation}
with the $O(x^{q-1})$ that is uniform on $M>0$.
\end{lem}

\begin{proof}
For $y\ge M$ we have $f_{pert}(y)=e^{-C_{q}y}$. Then when $x\le -M$, we can easily compute the convolution integral after shifting the variable:
\begin{align*}
\int_{-x}^{+\infty}(x+y)^{q-1}e^{-C_{q}y}dy&=
e^{C_{q}x}\int_0^{+\infty}y^{q-1}e^{-C_{q}y}dy\\
&=\frac{\Gamma(q)}{C_{q}^q}e^{C_{q}x}=\frac{\Gamma(q)}{C_{q}^q}f_{pert}(x),
\end{align*}
whence the first equality is proved.
Observe that the same calculation gives that for any $|x|\le M$,
\[
\int_{-x}^{+\infty}(x+y)^{q-1}f_{pert}(y)dy=
\frac{\Gamma(q)}{C_{q}^q}e^{-C_{q}M}.
\]
When $x>M$, we split the convolution integral at $-M$ and find
\[
\int_{-x}^{+\infty}(x+y)^{q-1}f_{pert}(y)dy=
\frac{\Gamma(q)}{C_{q}^q}e^{-C_{q}M}+\int_{-x}^{-M}(x+y)^{q-1}e^{C_{q}y}dy.
\]
The first term is a constant (and goes to zero as $M\to+\infty$). For the second term, by monotonicity we have 
\begin{align*}
\int_{-x}^{-M}(x+y)^{q-1}e^{C_{q}y}dy &\le (x-M)^{q-1}\int_{-x}^{-M}e^{C_{q}y}dy\\
&\le x^{q-1}\frac{1}{C_{q}}\left(e^{-C_{q}M}-e^{-C_{q}x}\right)\le \frac{e^{-C_{q}M}}{C_{q}}x^{q-1},
\end{align*}
concluding the proof of the lemma.
\end{proof}

\begin{lem}\label{l:ineq_P}
For any positive $\lambda$, $\lambda_{pert}$ and $\delta>0$, the function $P_k$ verifies
\begin{equation}\label{eq:Pk_sum}
 P_k(\lambda+\delta\,\lambda_{pert})\ge e^{-\delta\,\lambda_{pert}} P_k(\lambda)
\end{equation}
and
\begin{equation}\label{eq:Pk_sum2}
 P_k\left((\lambda-\delta\,\lambda_{pert})^+\right)\le e^{\delta\,\lambda_{pert}} P_k(\lambda),
\end{equation}
where $(\lambda-\delta\,\lambda_{pert})^+=\max\left(\lambda-\delta\,\lambda_{pert},0\right)$.
\end{lem}

The proof of Lemma~\ref{l:ineq_P} is a straightforward estimate. We are now ready to conclude with the proof of Lemma~\ref{l:last1}:

\begin{proof}[Proof of Lemma~\ref{l:last1}]
We claim that it is enough to consider $C_{q}$ so large that
\begin{equation}\label{eq:constant_C}
\frac{\Gamma(q)}{C_{q}^q}\le \frac{1}{4}.
\end{equation}

We first prove \ref{l:last1a}).
At the level of tail distribution functions, the operator $\Phi$ acts by $f\mapsto P_k\circ I[f]$ and reverses the order relations. We remark that the operator $f\mapsto P_k\circ I[f]$ is well-defined and order-reversing even if $f$ is not a tail distribution function. Therefore from~\eqref{eq:upper_class} one has the inequality
\begin{equation}
\label{eq:in-phi}
f_{\Phi[\mu]}(x)\ge P_k\left (I\left [f_{\bm}+\delta f_{pert}\right ](x)\right )\quad\text{for every }x\in\R. 
\end{equation}
Observe that the function $I$ is defined by a convolution and is therefore linear. 

Recall that $f_{\bm}$ satisfies $f_{\bm}=P_k\circ I[f_{\bm}]$. Then, considering $\lambda=I[f_{\bm}]$ and $\lambda_{pert}=I[f_{pert}]$ in Lemma~\ref{l:ineq_P} as functions, using~\eqref{eq:Pk_sum} in~\eqref{eq:in-phi}, one has
\begin{align*}
f_{\Phi[\mu]}(x)&\ge e^{-\delta\,\lambda_{pert}(x)}f_{\bm}(x)\\
&\ge \left (1-\delta\,\lambda_{pert}(x)\right )f_{\bm}(x)\quad\text{for every }x\in\R.
\end{align*}
Therefore, we have the inequality in \ref{l:last1a}) as soon as we prove that there exist $M$ and $K>0$ such that
\begin{equation}\label{eq:in-mu}
\lambda_{pert}(x)f_{\bm}(x)\le
\begin{cases}
K&\text{for every }|x|< M,\\
\frac{1}{2} f_{pert}(x)&\text{for every }|x|\ge M.
\end{cases}
\end{equation}
The first bound follows simply by compactness (in fact, from the proof of Lemma~\ref{l:pert} and the choice~\eqref{eq:constant_C}, we can take $K=\frac{1}{4}e^{-C_qM}$). For the second inequality we need to proceed differently, according to the sign of $x$. This is due to the fact that the asymptotic behaviours of the function $\lambda_{pert}=I[f_{pert}]$, at $+\infty$ and $-\infty$, are different (Lemma~\ref{l:pert}).

When $x\le -M$, we can plainly bound $f_{\bm}(x)$ by $1$
and then, by the equality \eqref{eq:mu1} from Lemma~\ref{l:pert}, we have the upper bound \eqref{eq:in-mu}, provided $C_{q}$ satisfies \eqref{eq:constant_C}.

On the other hand, as $x\to+\infty$ we have that the tail distribution function $f_{\bm}$ decays sufficiently fast compared to the other functions: using \eqref{eq:tail_cond1} from Lemma~\ref{l:td} jointly with \eqref{eq:mu2} from Lemma~\ref{l:pert}, one has
\[
\lambda_{pert}(x)f_{\bm}(x)=O\left (x^{qk-1}e^{-x^q}\right ),
\]
while $f_{pert}(x)=e^{-C_{q}x}$. Thus the second inequality~\eqref{eq:in-mu} is satisfied as soon as $M$ is sufficiently large. This proves~\ref{l:last1a}).

\smallskip

We now focus on \ref{l:last1b}). 
Since $f_\mu$ is a tail distribution function, it cannot be negative and 
therefore the inequality \eqref{eq:lower_class} is actually stronger:
\[
f_\mu(x)\ge (f_{\bm}(x)-\delta\,f_{pert}(x))^+\quad\text{for all }x\in\R.
\]
Proceeding as above, and using the same notations, we get
\[
f_{\Phi[\mu]}(x)\le P_k\left(\left(\lambda(x)-\delta \lambda_{pert}(x)\right)^+\right)\quad\text{for every }x\in\R
\]
and thus, by \eqref{eq:Pk_sum2} in Lemma~\ref{l:ineq_P}, we have
\begin{align*}
f_{\Phi[\mu]}(x)&\le e^{\delta\,\lambda_{pert}(x)}f_{\bm}(x)\\
&=f_{\bm}(x)+\left(e^{\delta\,\lambda_{pert}(x)}-1\right)f_{\bm}(x)\quad\text{for every }x\in\R.
\end{align*}
We need to prove that for $M$ large enough the following condition is satisfied: there exists $K>0$ such that for every $\delta\in (0,1)$
\begin{equation}\label{eq:in-mu2}
\left(e^{\delta\,\lambda_{pert}(x)}-1\right)f_{\bm}(x)\le 
\begin{cases}
K\delta&\text{for every }|x|<M,\\
\frac{\delta}{2} f_{pert}(x)&\text{for every }|x|\ge M.
\end{cases}
\end{equation}
The first inequality follows again by compactness. Let $x\le -M$; by the choice \eqref{eq:constant_C} of $C_{q}$, using \eqref{eq:mu1}, we observe that $\delta\,\lambda_{pert}(x)$ is upper bounded by~$\frac{1}{4}$. Then, for such values of $x$, we have
\begin{align*}
\left(e^{\delta\,\lambda_{pert}(x)}-1\right)f_{\bm}(x) &\le 2\delta\,\lambda_{pert}(x)f_{\bm}(x)\\
&\le 2\delta\,\lambda_{pert}(x)\le \frac{\delta}{2}f_{pert}(x)
\end{align*}
as desired.
On the other hand, as $x\to+\infty$, using \eqref{eq:tail_cond2} from Lemma~\ref{l:td} jointly with \eqref{eq:mu2} from Lemma~\ref{l:pert}, one has
\[
\left(e^{\delta\,\lambda_{pert}(x)}-1\right ) f_{\bm}(x)
=o\left (e^{-\frac{1}{2}x^q}\right ),
\]
while $f_{pert}(x)=e^{-C_{q}x}$. Thus the second inequality~\eqref{eq:in-mu2} is satisfied as soon as $M$ is sufficiently large. This proves~\ref{l:last1b}).
\end{proof}

We have concluded the proof of Lemma~\ref{l:last1}, and thus of Proposition~\ref{l:last}.

\section*{Concluding remarks}

\begin{rem}
The rather elementary strategy for endogeny that we have shown here seems to require a little ``complexity'' of the models. For the problem of random metrics (Theorem~\ref{t:applying}), we require the graph to be non-pivotal; for the mean-field optimization problems (Theorem~\ref{t:TSP}) we require the pseudo-dimension to be greater than~$1$. If these conditions are not satisfied, the stationary measures for the associated RDEs have exponential tail decays, which do not guarantee the assumption of superexponential decay of the center \ref{a:superexp}.
This last condition is essential in our proof of Lemma~\ref{l:7}, with which we control the displacement of the center (Lemma~\ref{l:7}).
It seems that with the simple exponential tail decay, our method should be completed with a finer asymptotic analysis. The result of~\cite{B} suggests the endogeny must hold also in these cases.
\end{rem}

\begin{rem}
It seems very plausible that with the employment of the adapted classes (Definition~\ref{d:adapted}) the arguments of \cite{BP} can be carried out to give an alternative proof of Theorem~\ref{t:BP}.
\end{rem}
\section*{Acknowledgements}
The authors thank Mikhail Khristoforov for the discussions on the cut-off method arisen during their common work \cite{KKT}, and David Aldous for his suggestions.
This work has been written during a visit of V.K.~to the Universidade Federal Fluminense supported by the R\'eseau France-Br\'esil in Mathematics.
V.K.~thanks also the support of the RFBR project 16-01-00748-a, the hospitality of University of Auckland, and both authors acknowledge the hospitality of CIRM.

\end{document}